\documentclass[11pt]{article}

\usepackage[english]{babel}
\usepackage[utf8x]{inputenc}
\usepackage[T1]{fontenc} 
\usepackage{scalerel}
\usepackage[dvipsnames]{xcolor}
\usepackage[margin=1.1in]{geometry}

\usepackage[math]{kurier}
\usepackage[sc]{mathpazo}

\usepackage{bbm}
\usepackage{xspace}
\usepackage{algorithmic}
\usepackage{algorithm}
\usepackage{graphicx}
\usepackage{subcaption}

\usepackage{url}            
\usepackage{hyperref}

\definecolor{darkpastelgreen}{rgb}{0.01, 0.75, 0.24}
\definecolor{cadmiumgreen}{rgb}{0.0, 0.42, 0.24}
\definecolor{armygreen}{rgb}{0.29, 0.33, 0.13}
\hypersetup{colorlinks,linkcolor={blue},citecolor={darkpastelgreen},urlcolor={red}}  

\usepackage{amsmath,amsfonts,amsthm,amssymb,color,newclude}
\usepackage{pifont}
\usepackage{natbib}

\usepackage[english]{babel}
\usepackage[utf8x]{inputenc}
\usepackage[T1]{fontenc} 
\usepackage{scalerel}

\usepackage{amsmath}
\usepackage{amsfonts}
\usepackage{amssymb}
\usepackage{comment}
\usepackage{amsthm}
\DeclareMathOperator*{\argmax}{arg\,max}

\newtheorem{assumption}{Assumption}
\numberwithin{assumption}{section}
\newtheorem{remark}{Remark}
\numberwithin{remark}{section}
\newtheorem{theorem}{Theorem}
\numberwithin{theorem}{section}

\newtheorem{lemma}[theorem]{Lemma}

\vspace{5mm}
\title{\textbf{Sharpened Quasi-Newton Methods: Faster Superlinear Rate and Larger Local Convergence Neighborhood}
\vspace{3mm}}

\author{Qiujiang Jin\thanks{Department of Electrical and Computer Engineering, The University of Texas at Austin, Austin, TX, USA. \{qiujiang@austin.utexas.edu\}.}\qquad Alec Koppel\thanks{Amazon, Bellevue, WA, USA. \{aekoppel@amazon.com\}.}\qquad Ketan Rajawat\thanks{Department of Electrical Engineering, Indian Institute of Technology Kanpur, Kanpur, UP, INDIA. \{ketan@iitk.ac.in\}.} \qquad   Aryan Mokhtari\thanks{Department of Electrical and Computer Engineering, The University of Texas at Austin, Austin, TX, USA. \{mokhtari@austin.utexas.edu\}.} \qquad    }

\date{\empty}

\begin{document}

\maketitle

\begin{abstract}
\noindent Non-asymptotic analysis of quasi-Newton methods have gained traction recently. In particular, several works have established a non-asymptotic superlinear rate of $\mathcal{O}((1/\sqrt{t})^t)$ for the (classic) BFGS method by exploiting the fact that its error of Newton direction approximation approaches zero. Moreover, a greedy variant of BFGS was recently proposed which accelerates its convergence by directly approximating the Hessian, instead of the Newton direction, and achieves a fast local quadratic convergence rate. Alas, the local quadratic convergence of Greedy-BFGS requires way more updates compared to the number of iterations that BFGS requires for a local superlinear rate. This is due to the fact that in Greedy-BFGS the Hessian is directly approximated and the Newton direction approximation may not be as accurate as the one for BFGS. In this paper, we close this gap and present a novel BFGS method that has the best of both worlds in that it leverages the approximation ideas of both BFGS and Greedy-BFGS to properly approximate the Newton direction and the Hessian matrix simultaneously. Our theoretical results show that our method out-performs both BFGS and Greedy-BFGS in terms of convergence rate, while it reaches its quadratic convergence rate with fewer steps compared to Greedy-BFGS. Numerical experiments on various datasets also confirm our theoretical findings.

\vspace{5mm}

\noindent \textbf{Keywords:} quasi-Newton method, BFGS method, greedy quasi-Newton method, superlinear convergence rate, non-asymptotic analysis
\end{abstract}

\newpage

\section{Introduction}
In this paper, we focus on the use of quasi-Newton methods to solve the following unconstrained problem
\begin{equation}\label{main_prob}
\min_{x \in \mathbb{R}^d} f(x),
\end{equation}
where $f:\mathbb{R}^d\to \mathbb{R}$ is strongly convex and its gradient is Lipschitz continuous; see details in Section~\ref{sec:general}. We denote the  unique optimal solution of \eqref{main_prob} by $x_*$.

First-order algorithms, i.e., gradient-based methods, are widely used for solving~\eqref{main_prob}, and it is well-known that their iterates converge to $x_*$ at a linear rate (i.e., the error decays exponentially fast). 
A major advantage of first-order methods is their low computational cost of  $\mathcal{O}(d)$, where $d$ is the problem dimension. However, the convergence rate of these methods depends on the problem curvature and hence they could be slow in ill-conditioned problems. Second-order methods that leverage the objective function Hessian to improve their curvature estimation often arise as a natural alternative to accelerate convergence in ill-posed problems, and they achieve fast local convergence rates \citep{bennett1916newton,ortega1970iterative,conn2000trust,nesterov2006cubic}. Specifically, Newton's method achieves a local quadratic convergence rate when applied to solve \eqref{main_prob} with the additional assumption that the Hessian is Lipschitz \citep[Chapter 9]{boyd04}. A major obstacle in the implementation of Newton's method though is its requirement to solve a linear system at each iteration, which makes its computational cost $\mathcal{O}(d^3)$.

Quasi-Newton (QN) methods serve as a middle ground between first- and second-order methods, as they improve the linear rate of first-order methods and converge superlineraly, and simultaneously their computation cost is $\mathcal{O}(d^2)$ which improves the $\mathcal{O}(d^3)$ cost of Newton-type methods. Their main idea is to construct a positive definite matrix that approximates the Hessian required in Newton's method. Since the update of Hessian approximation matrix in QN methods only requires a set of matrix-vector multiplications, their computational cost per iteration is $\mathcal{O}(d^2)$. There are several types of QN methods that differ in their Hessian approximation updates, including Symmetric Rank-One (SR1) method \citep{conn1991convergence}, the Broyden method \citep{broyden1965class,Broyden,gay1979some}, the Davidon-Fletcher-Powell (DFP) method \citep{davidon1959variable,fletcher1963rapidly}, the Broyden-Fletcher-Goldfarb-Shanno (BFGS) method \citep{broyden1970convergence,fletcher1970new,goldfarb1970family,shanno1970conditioning}, the limited-memory BFGS (L-BFGS) method \citep{nocedal1980updating,liu1989limited}, and the Greedy-QN method \citep{rodomanov2020greedy}. 

\begin{figure}
  \centering
    \includegraphics[width=0.45\linewidth]{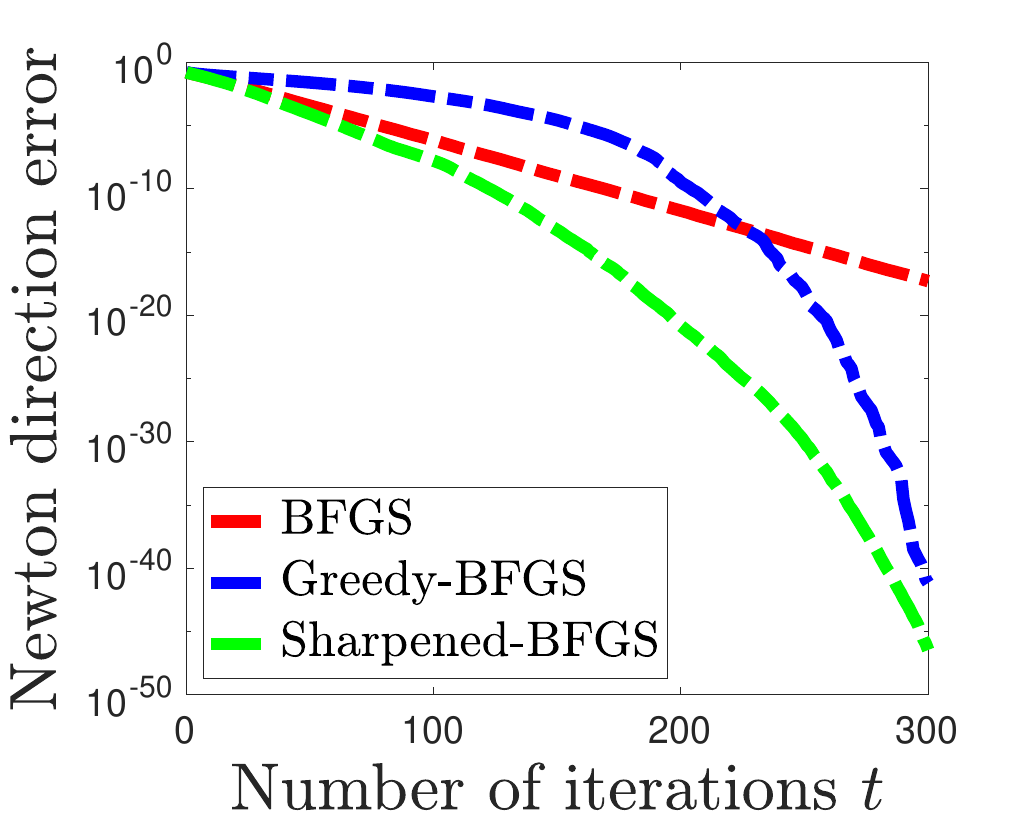}\qquad
    \includegraphics[width=0.45\linewidth]{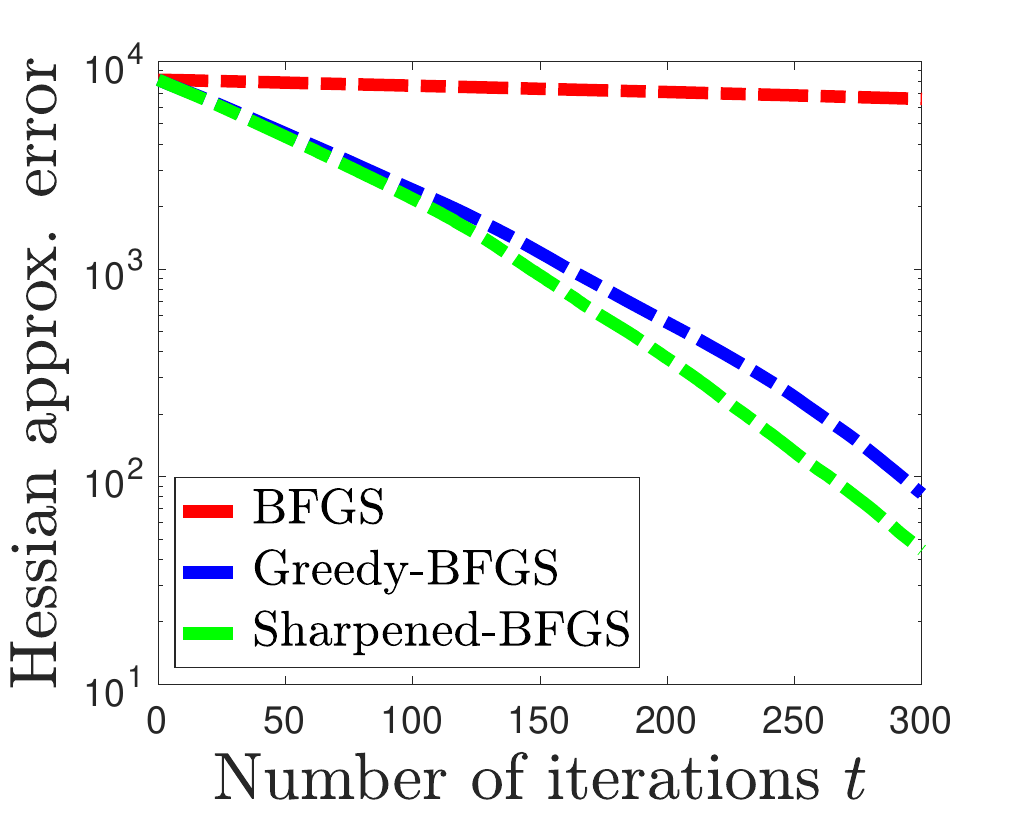}
  \vspace{-2mm}
  \caption{Comparison of BFGS, Greedy-BFGS, and the proposed Sharpened-BFGS algorithms in terms of Newton direction error~(top) and Hessian approximation error~(bottom) for a quadratic problem with dimension $d=400$ and condition number $\kappa=100$.}
  \label{toy_example}
\end{figure}

Perhaps the most important property of QN methods is their local superlinear convergence. Specifically, \citet{rodomanov2020greedy} introduced and analyzed a novel Greedy-QN method which is based on the classical Broyden class of QN methods and uses a greedily-selected vector to maximize certain measure of progress (see Section~\ref{sec:qn_method} for more details). Greedy-QN achieves a non-asymptotic quadratic convergence rate of $(1 - {1}/{d\kappa})^{t^2/2}(d\kappa)^t$, where $\kappa$ is the problem condition number. Note that this bound is equivalent to $((1 - {1}/{d\kappa})^{t/2}(d\kappa))^t$ which
shows that the fast quadratic convergence starts when $t\geq d\kappa \ln (d\kappa)$. It is also worth noting that in comparison with standard QN methods, greedy QN requires more information, including the diagonal elements of the Hessian at each iteration. 
In a follow-up work,  \citet{rodomanov2020rates} proved a non-asymptotic superlinear convergence rate for standard QN methods including the DFP and BFGS methods. They showed that BFGS and DFP achieve a local superlinear convergence rate of $({d\kappa}/{t})^{t/2}$ and $({d\kappa^2}/{t})^{t/2}$, respectively, under the assumptions that the objective function is strongly convex, smooth and strongly self-concordant. Later, \citet{rodomanov2020ratesnew} improved their results to the convergence rates of $({(d\ln{\kappa})}/{t})^{{t}/{2}}$ for BFGS and $({(d\kappa\ln{\kappa})}/{t})^{{t}/{2}}$ for DFP.  As noted in Table~\ref{tab_1}, the convergence rate of BFGS is slower than the one for Greedy-BFGS, but the superlinear rate starts at a smaller time index. 

\textbf{Contributions.} As mentioned above, (standard) BFGS aims at approximating the Newton direction and obtains a fast convergence rate from the beginning, but it fails to perfectly approximate the Hessian. On the other hand, Greedy-BFGS's goal is to directly approximate the Hessian matrix and therefore at first its convergence is slower than BFGS, but once its Hessian approximation improves it converges substantially faster than BFGS; see Figure~\ref{toy_example}. Considering these points, a natural question that arises is:

\vspace{-2mm}
\begin{quote}
\textit{Is it possible to achieve the best of two worlds and develop a QN method that  exhibits faster local convergence by approximating both the Newton direction and the Hessian matrix?}    
\end{quote}
\vspace{-2mm}

In this paper, we address this question by proposing a novel Sharpened-BFGS method which utilizes ideas of the classic BFGS and Greedy-BFGS. The proposed Sharpened-BFGS method exploits the initial fast convergence of BFGS by approximating the Newton direction, while developing an accurate approximation of the Hessian by following the Greedy-BFGS idea which allows for a quadratic convergence rate. As stated in Table~\ref{tab_1}, our method outperforms both BFGS and Greedy-BFGS in terms of convergence rate, while it reaches the superlinear convergence rate with fewer steps compared to the greedy method. We should add that the computational cost per iteration of Sharpened-BFGS is the same as its standard and greedy counterparts.

\noindent{\textbf{Related Work.}} \citet{qiujiang2020quasinewton1} established a non-asymptotic superlinear convergence rate of $({1}/{t})^{t/2}$ for standard QN methods under the assumptions that the objective function is strongly convex, smooth and its Hessian is Lipschitz continuous at the optimal solution. They also established a similar result for self-concordant functions. Their local convergence rate does not depend on the problem parameters such as $d$ or $\kappa$, but their results require both Hessian approximation error and the distance to the optimal solution to be sufficiently small. Moreover, \citet{zhangzhihua2021quasinewton1} obtained the explicit local superlinear convergence rate of the SR1 method. Further, \citet{zhangzhihua2021quasinewton2} extended the non-asymptotic local superlinear convergence rate of the Broyden family QN methods for solving nonlinear equations. It is worth noting that recently, \citet{zhangzhihua2021quasinewton3} proposed a randomized version of Greedy-BFGS which obtains a convergence rate of $(d\kappa(1 \!-\! \frac{1}{d})^{\frac{t}{2}})^{t}$. This randomized technique can be also utilized for our proposed Sharpened BFGS method to improve its convergence rate dependency in terms of $\kappa$. Due to space limitation, we present and analyze randomized Sharpened-BFGS in Appendix~\ref{sec:random}. 

\renewcommand{\arraystretch}{1.2}
\begin{table}[t]
  \vspace{1mm}
  \centering
  \begin{tabular}{ |c|c|c| }
    \hline
    Algorithm & Superlinear Rate & $t_0$ \\
    \hline
    \hline
    Standard BFGS & $(\frac{d\ln{\kappa}}{t})^{\frac{t}{2}}$ & $d\ln{\kappa}$ \\
    \hline
    Greedy-BFGS & $\left(d\kappa(1 \!-\! \frac{1}{d\kappa})^{\frac{t}{2}}\right)^{t}$ & $d\kappa\ln{(d\kappa)}$ \\
    \hline
    Sharpened-BFGS& $(1 - \frac{1}{d\kappa})^{\frac{t(t - 1)}{4}} (\frac{d\kappa}{t})^{\frac{t}{2}}$ & $d\kappa$ \\
    \hline
  \end{tabular}
  \caption{Convergence rate comparison of different variants of BFGS, where $d$ is the dimension, $\kappa$ is the condition number and $t_0$ is the time index at which the  superlinear convergence begins.}
  \label{tab_1}
  \vspace{-2mm}
\end{table}

\section{Preliminaries}\label{sec:qn_method}

In this section, we review some basics of QN methods that we require for developing our method. Consider $x_t\in \mathbb{R}^d$ as the iterate associated with time index $t$ and $\nabla f(x_t)\in \mathbb{R}^d$ as the objective function gradient evaluated at $x_t$. The general form of a QN update is given by
\begin{equation}\label{qn_method}
    x_{t + 1} = x_t - \eta_t G_t^{-1} \nabla{f(x_t)},
\end{equation}
where $\eta_t > 0$ is the step size (learning rate) and $G_t\in \mathbb{R}^{d\times d}$ is the matrix approximating the Hessian $\nabla^2 f(x_t)\in \mathbb{R}^{d\times d}$. In general, $\eta_t$ is determined by some line search algorithms so that the iteration generated converge to the optimal solution globally. In this paper, we focus on the local convergence analysis of QN algorithms, which requires the use of a unit step size $\eta_t = 1$. Hence, in the rest of the paper, we assume that the iterates $\{x_t\}_{t = 1}^{\infty}$ stay in a local neighborhood of $x_*$ and  $\eta_t=1$ is always admissible.

\subsection{BFGS Operator and Algorithm}\label{sec:standard_BFGS}

The essence of a QN method is its update for the Hessian approximation matrix $G_t$. There are various ways for updating $G_t$, but in this paper we focus on the BFGS method. Before stating the BFGS method, we first introduce it as an algorithm for approximating linear operators. This perspective turns out to be advantageous for unifying it with its greedy variant. To do so, consider $ A \in \mathbb{R}^{d \times d}$ as a positive
definite linear operator, and suppose $ G \in \mathbb{R}^{d \times d}$ is the operator that approximates $A$ and is updated according to the BFGS update. Then, the BFGS update rule for approximating operator $A$ along the direction $u \in \mathbb{R}^{d}\backslash\{0\}$ is 
\begin{equation}\label{BFGS_update}
    BFGS(A, G, u) = G_+ := G - \frac{G u u^\top G}{u^\top G u} + \frac{A u u^\top A}{u^\top A u}.
\end{equation}
Note that this update tries to move from $G$ to $G_+$ in a way that operators $A$ and $G_+$ are equal to each other in the direction of vector $u$, i.e., $Au=G_+u$.

\begin{remark}\label{remark_1}
As noted in \eqref{qn_method}, we need to compute the inverse of the Hessian approximation matrix at each step. Hence, we need a direct update for the Hessian inverse approximation matrices. By exploiting the Sherman-Morrison-Woodbury formula, one can show that the Hessian inverse approximation matrix $H = G^{-1}$ update can be written as
\begin{equation}\label{BFGS_inverse_update}
    H_+ = \left(I-\frac{u u^\top A}{u^\top Au}\right) H \left(I-\frac{ Au u^\top}{u^\top Au}\right) +\frac{u u^\top}{u^\top Au}.
\end{equation}
Hence, the computational cost of BFGS is $\mathcal{O}(d^2)$, as it only requires computation of matrix-vector multiplication.
\end{remark}

When we focus on minimizing a function and the ultimate linear operator that we aim to approximate is its curvature, then we select the direction as $u =x_{t+1}-x_t$ and the desired operator as the average Hessian $A = J_t:= \int_{0}^{1}\nabla^2{f(x_t + \tau(x_{t+1} - x_t))}d\tau$. This way we ensure that the new Hessian approximation matrix $G_{t+1}$ satisfies the secant condition, i.e.,  $$G_{t+1} (x_{t+1}-x_t)=J_t(x_{t+1}-x_t) = \nabla f(x_{t+1})-\nabla f(x_t), $$
If we define the variable and gradient differences as 
\begin{equation}\label{secant_condition}
    s_t := x_{t+1} - x_t, \qquad y_{t} := \nabla{f(x_{t+1})} - \nabla{f(x_t)},
\end{equation}
then the classic BFGS update is equivalent to 
\begin{equation}\label{BFGS_standard_update}
    G_{t+1} = G_t - \frac{G_t s_t s_t^\top G_t}{s_t^\top G_t s_t} + \frac{y_t y_t^\top}{s_t^\top y_t}.
\end{equation}

A major advantage of the BFGS update in \eqref{BFGS_standard_update} is that it forces the new Hessian approximation matrix $G_{t+1}$ to satisfy the secant condition, which implies $G_{t+1}s_t = y_t$. This condition ultimately ensures that the BFGS direction $G_t^{-1}\nabla f(x_t)$ approaches the Newton direction $\nabla^2 f(x_t)^{-1}\nabla f(x_t)$; see Chapter 6 of \cite{nocedal2006numerical} for details.

\subsection{Greedy-BFGS Algorithm}\label{sec:greedy_BFGS}

As mentioned in the previous section, BFGS does a good a job in approximating the Newton direction, but its Hessian approximation may not approach the true Hessian. To be precise, consider the following metric which captures the difference between positive definite matrices $A,G\in  \mathbb{R}^{d \times d}$
\begin{equation}\label{sigma}
\sigma(A, G) := Tr(A^{-1}G) - d,
\end{equation}
where 
$Tr(X)$ is the trace of matrix $X$, i.e., the sum of the diagonal elements of $X$. 
Note that if $A \preceq G$, we can use $\sigma(A, G)$ as a  potential function that measures the distance between two matrices $A$ and $G$. Note that $\sigma(A, G) = 0$ if and only if $A = G$. 
Using the above potential function, in the next lemma, we state the error of Hessian approximation for the BFGS operator in \eqref{BFGS_update}. The proof can be found in \cite{rodomanov2020greedy}.

\begin{lemma}\label{lemma_BFGS_general}
Consider positive definite matrices $A, G \in \mathbb{R}^{d \times d}$ and suppose that $G_{+} = BFGS(A, G, u)$ as defined in \eqref{BFGS_update} and $u \in \mathbb{R}^d\backslash\{0\}$. If $A \preceq G$, then we have
\begin{equation}\label{lemma_BFGS_general_1}
    \sigma(A, G) - \sigma(A, G_+) \geq \frac{u^\top G u}{u^\top A u}-1.
\end{equation}
\end{lemma}

This result shows how fast the gap between the Hessian approximation and the true Hessian decreases after one step of BFGS.
The result in  Lemma~\ref{lemma_BFGS_general} also shows that the selection of direction $u$ can influence the decrease in the trace potential function $\sigma(A, G)$ after one BFGS update. Note that for an arbitrary direction $u \in \mathbb{R}^{d}\backslash\{0\}$, there is no guarantee that the Hessian approximation matrix converges to the exact Hessian matrix. In fact, if we set $u=x^+-x$ as done in the classic BFGS update, there is no guarantee that $\sigma(A, G)$ converges to $0$. This observation reveals the following question: How can we select $u$ to maximize the progress in decreasing $\sigma(A, G)$ and ensuring that $\sigma(A, G)$ converges to $0$, i.e., $G$ converges to $A$? 

\citet{rodomanov2020greedy} answered this question by proposing a greedy selection scheme for determination of the best choice of $u$. To better explain this concept, consider a quadratic problem, where the objective function Hessian is fixed and denoted by the positive definite matrix $A$. In this case, to maximize the right hand side of \eqref{lemma_BFGS_general_1}, which shows the progress for the BFGS update, one could select $u$ as
\begin{equation}\label{greedy_vector}
    \bar{u}(A, G) := \argmax_{u \in \{e_i\}_{i = 1}^{d}} \frac{u^\top G u}{u^\top A u},
\end{equation}
where $\{e_i\}$ is the vector whose $i$-th element is $1$ and its remaining elements are $0$. If we choose $u = \bar{u}(A, G)$ in each iteration of BFGS update \eqref{BFGS_update}, we obtain the Greedy-BFGS algorithm in  \cite{rodomanov2020greedy}. The advantage of this greedily selected is that it ensures the trace potential function $\sigma(A, G)$ is strictly decreasing and converges to $0$ linearly as specified in the following lemma.
\begin{lemma}[\cite{rodomanov2020greedy}]\label{lemma_BFGS_greedy}
Consider positive definite matrices $A, G \in \mathbb{R}^{d \times d}$ that satisfy $A \preceq G$ and  $ \mu I \preceq A \preceq LI$, where $0 < \mu \leq L$ are two constants. Suppose that $\bar{G}_{+} = BFGS(A, G, \bar{u}(A, G))$ where $\bar{u}(A, G) \in \mathbb{R}^d$ is  greedily selected as defined in \eqref{greedy_vector}. Then,
\begin{equation}\label{lemma_BFGS_greedy_1}
    \sigma(A, \bar{G}_{+}) \leq \left(1 - \frac{\mu}{dL}\right)\sigma(A, G).
\end{equation}
\end{lemma}

This result shows that by following the Greedy-BFGS update the error of Hessian approximation, in terms of the metric $\sigma(.,.)$ defined in \eqref{sigma}, converges to zero linearly and eventually the sequence of Hessian approximations approaches the true Hessian. Note that, for the non-quadratic case, a similar argument holds, but the algorithm should be slightly modified as the computation of the average Hessian $J_t$ is costly and instead one might use the current Hessian $\nabla^2 f(x_t)$. We discuss this point in detail in the following section, when we present our Sharpened-BFGS method.

\section{Sharpened-BFGS}\label{sec:main}

In this section, we propose the Sharpened-BFGS algorithm which benefits from the update of BFGS for Newton direction approximation and the Greedy-BFGS update to approximate the Hessian matrix. In a nutshell, the update of Sharpened-BFGS first adjusts the Hessian approximation according to the BFGS update by setting $u=x_{t+1}-x_t$, and then improves the Hessian approximation by following the greedy update, and selecting the vector $u$ in a greedy fashion. To introduce our method, we first focus on a quadratic program where the Hessian is fixed. We then build on our intuition from the quadratic case to develop the general version of our method for the problem in \eqref{main_prob}.

\subsection{Quadratic Programming}\label{sec:quadratic}

Consider a special case of \eqref{main_prob} where the objective function is quadratic and given by
\begin{equation}\label{quadratic}
    \min_{x \in \mathbb{R}^d} f(x) = \frac{1}{2}x^\top A x + b^\top x,
\end{equation}
where $A \in \mathbb{R}^{d \times d}$ is a symmetric positive definite matrix satisfying $\mu I \preceq A \preceq LI$ and $b \in \mathbb{R}^d$. The Sharpened-BFGS algorithm applied to \eqref{quadratic} is shown in Algorithm~\ref{algo_quadratic}. We observe that the proposed algorithm involves two BFGS updates per iteration. Intuitively, we improve the Hessian approximation along the classical BFGS direction and subsequently along the Greedy-BFGS direction. Notice that the initial Hessian approximation matrix is $G_0 = LI$. Hence, the initial Hessian inverse approximation matrix is simply $H_0 = ({1}/{L})I$. For the quadratic problem, the sequence generated by Sharpened-BFGS converges to the optimal solution globally, as we show in Theorems~\ref{thm_quadratic_linear} and \ref{thm_quadratic_superlinear}. Hence, the initial point $x_0$ can be any vector in $\mathbb{R}^d$.

To formally show how Sharpened-BFGS exploits the fast properties of both BFGS and Greedy-BFGS, we first define the Newton decrement as $\lambda_{f}(x) := \sqrt{\nabla{f(x)}^\top \nabla^2 f(x)^{-1} \nabla{f(x)}}$. In our results, we report convergence in terms of $   \lambda_{f}(x)$ and we use the notation $\lambda_t := \lambda_{f}(x_t)$. We next state the following intermediate result  that shows for the class of quasi-Newton updates defined in \eqref{qn_method} (with step size $\eta=1$) on a quadratic program, how fast  $\lambda_{f}(x)$ converges to zero. The proof of this result can be found in \cite{rodomanov2020rates}.

\begin{algorithm}[t]
\caption{Sharpened-BFGS applied to \eqref{quadratic}.}\label{algo_quadratic}
\begin{algorithmic}[1] 
{\REQUIRE Initial point $x_0$ and initial matrix $G_0 = LI$.
\FOR {$t = 0,1,2,\ldots$}
    \STATE Update the variable: $x_{t + 1} = x_t - G_t^{-1} \nabla{f(x_t)}$;
    \STATE Compute  $s_t = x_{t + 1} - x_t$;
    \STATE Compute $\bar{G_t} = BFGS(A, G_t, s_t)$;
    \STATE Compute  $\bar{u} = \bar{u}(A, \bar{G_t})$ according to \eqref{greedy_vector};
    \STATE Compute $G_{t + 1} = BFGS(A, \bar{G_t}, \bar{u})$;
\ENDFOR}
\end{algorithmic}\end{algorithm}

\begin{lemma}\label{lemma_quadratic_theta}
Consider the quadratic function in \eqref{quadratic} and the sequence of iterates generated according to the update in \eqref{qn_method} with step size $\eta_t=1$. 
 Then, we have that
\begin{equation}\label{lemma_quadratic_theta_1}
    \lambda_{t+1} = \theta(A, G_t, x_{t + 1} - x_{t})\lambda_{t},
\end{equation}
where 
\begin{equation}\label{theta}
\theta(A,G,u) := \left(\frac{u^\top (G - A) A^{-1} (G - A) u}{u^\top G A^{-1} G u}\right)^{\frac{1}{2}}.
\end{equation}
\end{lemma}

First, note that $\theta(A,G,u)$ captures the closeness of $G$ and $A$ along the direction of $u$, where $u \in \mathbb{R}^{d}\backslash\{0\}$. The above result shows that the contraction factor for the convergence of the Newton decrement is related to the gap between $G_t(x_{t+1}-x_t)$ and $A(x_{t+1}-x_t)$. In the following theorem, we characterize a global upper bound on $\theta(A, G_t, x_{t + 1} - x_{t})  $ for the Sharpened-BFGS method.

\begin{theorem}\label{thm_quadratic_linear}
Consider the Sharpened-BFGS method in Algorithm~\ref{algo_quadratic} applied to the quadratic problem~\eqref{quadratic}. Then,
\begin{equation}\label{thm_quadratic_linear_1}
\theta(A, G_t, x_{t + 1} - x_{t})\leq 1 - \frac{\mu}{L}, \qquad \forall t \geq 0,
\end{equation}
and therefore
\begin{equation}\label{thm_quadratic_linear_2}
    \lambda_t \leq \left(1 - \frac{\mu}{L}\right)^{t}\lambda_0, \qquad \forall t \geq 0.
\end{equation}
\end{theorem}

\begin{proof}
Check Appendix~\ref{sec:proof_of_thm_quadratic_linear}.
\end{proof}

The above result shows that the iterates generated by Sharpened-BFGS converge to the solution at a linear rate of $1 - {\mu}/{L}$. However, this is not a tight bound and simply follows from the fact that eigenvalues of $G_t$ and $A$ are uniformly bounded. In the next lemma, we present that the sequence $\theta(A, G_t, x_{t + 1} - x_{t})$ eventually approaches zero and hence the iterates of Sharpened-BFGS converge superlinearly.

\begin{lemma}\label{lemma_quadratic_potential}
Consider Sharpened-BFGS in Algorithm~\ref{algo_quadratic} applied to the quadratic function \eqref{quadratic}. Further, define $\theta_t:=\theta(A, G_t, x_{t + 1} - x_{t})$ and $\sigma_t := \sigma(A, G_t)$. Then, 
\begin{equation}\label{lemma_quadratic_potential_1}
    \sigma_{t+1} \leq \left(1 - \frac{\mu}{dL}\right)\left(\sigma_t - \theta_t^2\right)
\end{equation}
for any $t \geq 0$. Moreover, we have 
\begin{equation}\label{lemma_quadratic_potential_2}
 \sum_{i = 0}^{t - 1}\frac{\theta^2_i}{(1 - \frac{\mu}{d L})^{i}}  \leq \sigma_0, \qquad \forall t \geq 1. 
\end{equation}
\end{lemma}

\begin{proof}
Check Appendix~\ref{sec:proof_of_lemma_quadratic_potential}.
\end{proof}

First, note that \eqref{lemma_quadratic_potential_1} shows that in Sharpened-BFGS $\sigma_t$ converges to zero as in the Greedy-BFGS algorithm. Moreover comparing the bound in \eqref{lemma_quadratic_potential_1} with the one in \eqref{lemma_BFGS_greedy_1} shows that in Sharpened-BFGS $\sigma_t$ converges faster than Greedy-BFGS, as $\theta^2_t > 0$. Second, the result in \eqref{lemma_quadratic_potential_2} shows that the sequence $\theta_t$ converges to zero. Hence, one can leverage this result to show a tighter upper bound for  $\theta_t$ compared to the one in \eqref{thm_quadratic_linear_1} and show  a faster rate than the one in \eqref{thm_quadratic_linear_2} for the Sharpened-BFGS. This goal is accomplished in the following Theorem. 

\begin{theorem}\label{thm_quadratic_superlinear}
Consider Sharpened-BFGS described in Algorithm~\ref{algo_quadratic} applied to the quadratic function \eqref{quadratic}. Then, for $t\geq 1$ we have
\begin{equation}\label{thm_quadratic_superlinear_1}
    \lambda_t \leq \left(1 - \frac{\mu}{dL}\right)^{\frac{t(t - 1)}{4}} \left(\frac{dL}{t\mu}\right)^{\frac{t}{2}}\lambda_0.
\end{equation}
\end{theorem}

\begin{proof}
Check Appendix~\ref{sec:proof_of_thm_quadratic_superlinear}.
\end{proof}

If we analyze the superlinear convergence rate in \eqref{thm_quadratic_superlinear_1}, we observe that there are two terms that contribute to the  rate. The first is the quadratic rate $(1 - \frac{\mu}{dL})^{\frac{t(t - 1)}{4}}$ and the second is $(\frac{dL}{t\mu})^{\frac{t}{2}}$. Notice that for the second term $(\frac{dL}{t\mu})^{\frac{t}{2}}$, the superlinear convergence kicks in only after $t \geq d\frac{L}{\mu}$. Hence, by combining the results of Theorem~\ref{thm_quadratic_linear} and \ref{thm_quadratic_superlinear}, we obtain that during the initial iterations $t<d\frac{L}{\mu}$ Sharpened-BFGS converges linearly and for $t>d\frac{L}{\mu}$ the rate becomes faster than quadratic rate and $\lambda_t$ approaches zero at a rate of  $\mathcal{O}((1-\frac{\mu}{dL})^{t^2} (\frac{dL}{\mu t})^t)$.

\subsection{General Strongly-Convex and Smooth Setting}\label{sec:general}

\begin{algorithm}[t]
\caption{General Sharpened-BFGS }\label{algo_general} 
\begin{algorithmic}[1] 
{\REQUIRE Initial point $x_0$ and initial matrix $G_0 = LI$.
\FOR {$t = 0,1,2,\ldots$}
    \STATE Update $x_{t + 1} = x_t - G_t^{-1} \nabla{f(x_t)}$;
    \STATE Compute  $s_t = x_{t + 1} - x_t$;
    \STATE Set  $J_t = \int_{0}^{1}\nabla^2{f(x_t + \tau s_t)}d\tau$;
    \STATE Compute $\bar{G_t} = BFGS(J_t, G_t, s_t)$;
    \STATE Compute $r_t = \|x_{t + 1} - x_{t}\|_{x_t}$;
    \STATE Compute  $\hat{G_t} = (1 + {Mr_t}/{2})^2\bar{G_t}$;
    \STATE Compute  $\bar{u} = \bar{u}(\nabla^2{f(x_{t + 1})}, \hat{G_t})$ according to \eqref{greedy_vector};
    \STATE Compute $G_{t + 1} = BFGS(\nabla^2{f(x_{t + 1})}, \hat{G_t}, \bar{u})$;
\ENDFOR}
\end{algorithmic}\end{algorithm}

In this section, we extend our algorithm and its analysis to non-quadratic convex programs. To do so, We first state the required assumptions on the objective function to establish the superlinear convergence rate of Sharpened-BFGS.

\begin{assumption}\label{ass_str_cvx_smooth}
The objective function $f$ is twice differentiable. It is strongly convex with parameter $\mu > 0$ and its gradient $\nabla f$ is Lipschitz continuous with parameter $L > 0$.
\end{assumption}

\begin{assumption}\label{ass_str_concordant}
The objective function $f$ is strongly self-concordant with $M > 0$, i.e., for any $ x, y, z, w \in \mathbb{R}^{d}$, we have $\nabla^{2}{f(y)} - \nabla^{2}{f(x)} \preceq M\|y - x\|_{z}\nabla^{2}{f(w)}$, where 
$$
\|y - x\|_{z} : = \sqrt{(y - x)^\top \nabla^2{f(z)} (y - x)}.
$$
\end{assumption}

The strongly self-concordant functions form a subclass of the famous self-concordant functions class introduced in \citep{nesterov1989self,nesterov1994interior}, which plays a fundamental role in the local analysis of Newton's method. The concept of strong self-concordance was first proposed by \citet{rodomanov2020greedy} to establish the explicit quadratic convergence rate of the greedy QN method. Note that a strongly convex function with Lipschitz continuous Hessian is strongly self-concordant; see Example 4.1 in \cite{rodomanov2020greedy}. 

The general Sharpened-BFGS method is presented in Algorithm~\ref{algo_general}. We observe that Algorithm~\ref{algo_general} is fundamentally similar to the Algorithm~\ref{algo_quadratic} for the quadratic case, but there are still some differences between them. In general, similar to Algorithm~\ref{algo_quadratic}, we first update the Hessian approximation matrix along the standard BFGS direction and then along the Greedy-BFGS direction. The only difference between Algorithm~\ref{algo_quadratic} and \ref{algo_general} is that we add the correction term $r_t = \|x_{t + 1} - x_{t}\|_{x_t}$ in Steps $6$ and $7$ of Algorithm~\ref{algo_general}. The reason for this modification is that the trace potential function $\sigma(A, G)$ is only well-defined under the condition $A \preceq G$. Suppose currently the condition $\nabla^2{f(x)} \preceq G$ holds. We add the correction term to ensure that after one BFGS update, the new point $x_+$ and the new Hessian approximation matrix $G_+$ still satisfy the condition $\nabla^2{f(x_+)} \preceq G_+$. Since the Hessian of the general convex function is not fixed, there is no guarantee that the quasi-Newton update can preserve the property of $\nabla^2{f(x)} \preceq G$ without that correction term. The initial Hessian approximation matrix is still $G_0 = LI$. We should also add that Step $4$ does not require computing $J_t = \int_{0}^{1}\nabla^2{f(x_t + \tau s_t)}d\tau$ explicitly. As we discussed in Section~\ref{sec:standard_BFGS}, we can compute the standard BFGS update in Step $5$ according to \eqref{BFGS_standard_update}.

\begin{remark}\label{remark_5}
The computational cost per iteration of Algorithm~\ref{algo_quadratic} is $\mathcal{O}(d^2)$. The difference between Algorithm~\ref{algo_general} and \ref{algo_quadratic} is in Steps $6$ and $7$ of Algorithm~\ref{algo_general}. The computational cost of calculating the vector $r_t = \|x_{t + 1} - x_{t}\|_{x_t}$ and the matrix $\hat{G_t} = (1 + {Mr_t}/{2})^2\bar{G_t}$ is also $\mathcal{O}(d^2)$. Hence, the computational cost per iteration of Algorithm~\ref{algo_general} is also $\mathcal{O}(d^2)$.
\end{remark}

The convergence rate analysis of the Sharpened-BFGS method is inspired by the counterpart of the quadratic function, but there are still some differences between these two analyses as we need to take into account the variation of the Hessian for the non-quadratic case. Most importantly, for the general (non-quadratic) case, we can only obtain  local convergence results as we state in  Theorems~\ref{thm_general_linear} and \ref{thm_general_superlinear}. In other words, the initial point $x_0$ should be within a local neighborhood of the optimal solution $x_*$ to guarantee the  convergence of Sharpened-BFGS.
Similar to Section~\ref{sec:quadratic}, we first establish the relationship between $\theta$ defined in \eqref{theta} and the Newton decrement $\lambda_f(x)$ after one iteration of quasi-Newton update for minimizing a general convex function. The proof can be found in \cite{rodomanov2020rates}.

\begin{lemma}\label{lemma_general_theta}
Consider problem \eqref{main_prob} and suppose Assumptions~\ref{ass_str_cvx_smooth}--\ref{ass_str_concordant} are satisfied. Then, the iterates $x_t$ generated according to the update in \eqref{qn_method} with step size $\eta_t=1$ satisfy
\begin{equation}\label{lemma_general_theta_1}
    \lambda_{t+1} \leq \left(1 + \frac{Mr_t}{2}\right)\theta(J_t, G_t, x_{t + 1} - x_{t})\lambda_t,
\end{equation}
where $J_t := \int_{0}^{1}\nabla^2{f(x_t + \tau(x_{t + 1} - x_{t}))d\tau}$ and $r_t := \|x_{t+1} - x_t\|_{x_t}$. 
\end{lemma}

Notice that the above lemma is in parallel to Lemma~\ref{lemma_quadratic_theta} for the quadratic case. Now, we establish a local upper bound for the measurement $\theta(J_t, G_t, x_{t + 1} - x_{t})$ and prove the local linear convergence rate of the general Sharpened-BFGS method, which is similar to the results in Theorem~\ref{thm_quadratic_linear}.

\begin{theorem}\label{thm_general_linear}
Consider Sharpened-BFGS in Algorithm~\ref{algo_general} applied to the objective function $f$ satisfying Assumption~\ref{ass_str_cvx_smooth} and \ref{ass_str_concordant}. Moreover, suppose that the initial point $x_0$ satisfies
\begin{equation}\label{thm_general_linear_1}
\lambda_0 \leq \frac{C_0\mu}{ML}, 
\end{equation}
where  $C_0 = \frac{1}{4}\ln{\frac{3}{2}}$.
Then, for any $t \geq 0$ we have
\begin{equation}\label{thm_general_linear_2}
    \theta(J_t, G_t, x_{t + 1} - x_{t}) \leq 1 - \frac{2\mu}{3L},
\end{equation}
which leads to
\begin{equation}\label{thm_general_linear_3}
    \lambda_t \leq (1 - \frac{\mu}{2L})^{t}\lambda_0.
\end{equation}
\end{theorem}
\vspace{-2mm}
\begin{proof}
Check Appendix~\ref{sec:proof_of_thm_general_linear}.
\end{proof}

The above theorem presents that in a local neighborhood of the optimal solution, the iterates generated by Sharpened-BFGS achieve a linear convergence rate of $1 - {\mu}/{2L}$, which is obtained by a loose bound on $\theta_t$. As mentioned in Section~\ref{sec:quadratic}, our ultimate target is to improve this to the superlinear rate. In the following lemma, we establish inequalities similar to \eqref{lemma_quadratic_potential_1} and \eqref{lemma_quadratic_potential_2} to establish a superlinear convergence rate for Sharpened-BFGS.

\begin{lemma}\label{lemma_general_potential}
Consider Sharpened-BFGS in Algorithm~\ref{algo_general} applied to the objective function $f$ satisfying Assumptions~\ref{ass_str_cvx_smooth} and \ref{ass_str_concordant}. Moreover, suppose that the initial point $x_0$ satisfies 
\begin{equation}\label{lemma_general_potential_1}
\lambda_0 \leq \frac{C_0\mu}{ML}, 
\end{equation}
where $ C_0 = \frac{1}{4}\ln{\frac{3}{2}}$.
Further, consider the definitions  $\theta_t := \theta(\nabla^2{f(x_{t})}, G_t, x_{t + 1} - x_{t})$ and $\sigma_t := \sigma(\nabla^2{f(x_{t})}, G_{t})$. Then, for any $t \geq 0$ it holds that
\begin{equation}\label{lemma_general_potential_2}
    \sigma_{t+1} \leq (1 - \frac{\mu}{2dL})\left[(1 + \frac{M\lambda_t}{2})^4(\sigma_t + 4Md\lambda_t) - \frac{1}{4}\theta^2_t\right].
\end{equation}
Moreover, we have 
\begin{equation}\label{lemma_general_potential_3}
    \sum_{i = 0}^{t - 1}\frac{\theta^2_i}{(1 - \frac{\mu}{2d L})^{i}}  \leq 8(\sigma_0 + 4Md\lambda_0), \qquad \forall t \geq 1. 
\end{equation}
\end{lemma}
\vspace{-2mm}
\begin{proof}
Check Appendix~\ref{sec:proof_of_lemma_general_potential}.
\end{proof}

The conclusions of the above lemma are similar to the ones for the quadratic case in Lemma~\ref{lemma_quadratic_potential}, except that a local condition is required.
Specifically, the result in \eqref{lemma_general_potential_2} implies that in Sharpened-BFGS $\sigma_t$ converges to zero as long as $\lambda_0$ is sufficiently small. Similarly this rate is faster than the one for Greedy-BFGS as  $\theta^2_t > 0$. Second, the result in \eqref{lemma_general_potential_3} implies that the sequence $\theta_t$ converges to zero at a fast rate as its weighted sum by a factor larger than $1$ that is exponentially growing is finite. Hence, locally this result provides a tighter upper bound on $\theta_t$ compared to the one in \eqref{thm_general_linear_2} and shows  a faster rate than the one in \eqref{thm_general_linear_3} for Sharpened-BFGS. We leverage these points to establish the convergence rate of Sharpened-BFGS for non-quadratic problems.

\begin{theorem}\label{thm_general_superlinear}
Consider the Sharpened-BFGS method in Algorithm~\ref{algo_general} applied to the objective function $f$ satisfying Assumptions~\ref{ass_str_cvx_smooth}-\ref{ass_str_concordant}. Suppose the initial point $x_0$ satisfies
\begin{equation}\label{thm_general_superlinear_1}
\lambda_0 \leq \frac{C_1\mu}{dML}, 
\end{equation}
where $ C_1 = \frac{\ln{2}}{20}$. 
Then, $\forall t \geq 1$, we have
\begin{equation}\label{thm_general_superlinear_2}
    \lambda_t \leq 2\left(1 - \frac{\mu}{2dL}\right)^{\frac{t(t - 1)}{4}} \left(\frac{8dL}{t\mu}\right)^{\frac{t}{2}}\lambda_0.
\end{equation}

\end{theorem}

\begin{proof}
Check Appendix~\ref{sec:proof_of_thm_general_superlinear}.
\end{proof}

We observe that the superlinear convergence rate of Theorem~\ref{thm_general_superlinear} is very similar to the result of Theorem~\ref{thm_quadratic_superlinear}. As we discussed in the last paragraph of Section~\ref{sec:quadratic}, we can summarize the Theorem~\ref{thm_general_linear} and \ref{thm_general_superlinear} into one convergence result. Hence, the iteration generated by the Sharpened-BFGS method applied to the unconstrained optimization problem as specified in Algorithm~\ref{algo_general} satisfies the following local convergence rate. When the iteration number $t \leq \Theta(d\frac{L}{\mu})$, the linear convergence rate in \eqref{thm_general_linear_2} holds. When $t \geq \Theta(d\frac{L}{\mu})$, we can reach the superlinear convergence rate in \eqref{thm_general_superlinear_2}.

\section{Discussions}\label{sec:discussion}

In this section, we compare the convergence results of Sharpened-BFGS with the  ones for Greedy-BFGS and standard BFGS. We specifically focus on the case that the objective function satisfies Assumption~\ref{ass_str_cvx_smooth} and \ref{ass_str_concordant}. 
To simplify the comparisons, we replace all the universal constants with $1$ in the convergence results and only compare the parameters $\mu$, $L$, $M$ and $d$ defined in Assumption~\ref{ass_str_cvx_smooth} and \ref{ass_str_concordant}. We denote the condition number by $\kappa = {L}/{\mu} \geq 1$. 

\noindent \textbf{Sharpened-BFGS.} According to our result, if we set $G_0 = LI$ and the initial point $x_0$ satisfies
\begin{equation*}
    \lambda_f(x_0) = \mathcal{O}\left(\frac{1}{dM\kappa}\right),
\end{equation*}
then the iterates generated by Sharpened-BFGS satisfy:
\begin{equation*}
\frac{\lambda_f(x_t)}{\lambda_f(x_0)} \leq
\min\left\{\left(1 \!-\! \frac{1}{\kappa}\right)^t ,\  \left(1 \!-\! \frac{1}{d\kappa}\right)^{\!\frac{t(t - 1)}{4}}\!\! \left(\frac{d\kappa}{t}\right)^{\frac{t}{2}} \right\}.
\end{equation*}
Hence, for $t< d\kappa$, the first upper bound is smaller and the Newton decrement converges at a linear rate of $(1-\frac{1}{\kappa})^t$, and for $t\geq d\kappa $ the second term becomes smaller and we observe a superlinear rate of $(1 \!-\! \frac{1}{d\kappa})^{\frac{t(t - 1)}{4}} (\frac{d\kappa}{t})^{\frac{t}{2}}$, which is faster than quadratic rate.

\noindent \textbf{Greedy-BFGS.} Next, we present the convergence result for Greedy-BFGS in  \cite{rodomanov2020greedy}. If we set $G_0 = LI$ and the initial point $x_0$ satisfies
\begin{equation*}
    \lambda_f(x_0) = \mathcal{O}\left(\frac{1}{dM\kappa}\right),
\end{equation*}
then the iterates of Greedy-BFGS satisfy:
\begin{equation*}
\frac{\lambda_f(x_t)}{\lambda_f(x_0)} \leq
\min\left\{\left(1 \!-\! \frac{1}{\kappa}\right)^t ,\  \left(1 \!-\! \frac{1}{d\kappa}\right)^{\!\frac{t(t - 1)}{2}}\!\! \left(\frac{1}{2}\right)^{t} \right\}.
\end{equation*}

We observe that the superlinear convergence appears after $d\kappa\ln{(d\kappa)}$ iterations for Greedy-BFGS, while for Sharpened-BFGS it takes $d\kappa$ steps to reach the superlinear convergence. Hence, Sharpened-BFGS achieves the superlinear rate with fewer iterations compared to Greedy-BFGS. Moreover, eventually the superlinear convergence rate of Sharpened-BFGS is faster than the one for Greedy-BFGS. This is because both of the methods achieve a quadratic convergence rate of the form $(1 - \frac{1}{d\kappa})^{t^2}$. However, when $t$ is sufficiently large, we have $
(\frac{d\kappa}{t})^{\frac{t}{2}} \ll (\frac{1}{2})^{t}$.

\noindent \textbf{BFGS.} Now, we present the convergence result for BFGS provided in \cite{rodomanov2020ratesnew}. If we set $G_0 = LI$ and the initial point $x_0$ satisfies
\begin{equation*}
    \lambda_f(x_0) = \max\left\{\mathcal{O}\left(\frac{1}{M\kappa}\right), \mathcal{O}\left(\frac{1}{Md\ln{\kappa}}\right)\right\},
\end{equation*}
then the iterates of BFGS satisfy:
\begin{equation*}
\frac{\lambda_f(x_t)}{\lambda_f(x_0)} \leq
\min\left\{\left(1 \!-\! \frac{1}{\kappa}\right)^t ,\ \left(\frac{d\ln{\kappa}}{t}\right)^{\frac{t}{2}} \right\}.
\end{equation*}
We observe that the superlinear convergence of BFGS starts after $d\ln{\kappa}$ steps, while it takes $d\kappa$ iterations for the appearance of the superlinear convergence of Sharpened-BFGS. However, the superlinear convergence rate of Sharpened-BFGS is faster than BFGS as for large $t$ we have
\begin{equation*}
\left(1 \!-\! \frac{1}{d\kappa}\right)^{\!\frac{t(t - 1)}{4}}\!\! \left(\frac{d\kappa}{t}\right)^{\frac{t}{2}}
\ll 
\left(\frac{d\ln{\kappa}}{t}\right)^{\frac{t}{2}}.
\end{equation*}

For the broad strokes, see the quantitative comparisons summarized in Table \ref{tab_1}.

\section{Numerical Experiments}\label{sec:numerical}

\begin{figure}
\centering
\begin{subfigure}{0.32\textwidth}
    \includegraphics[width=\textwidth]{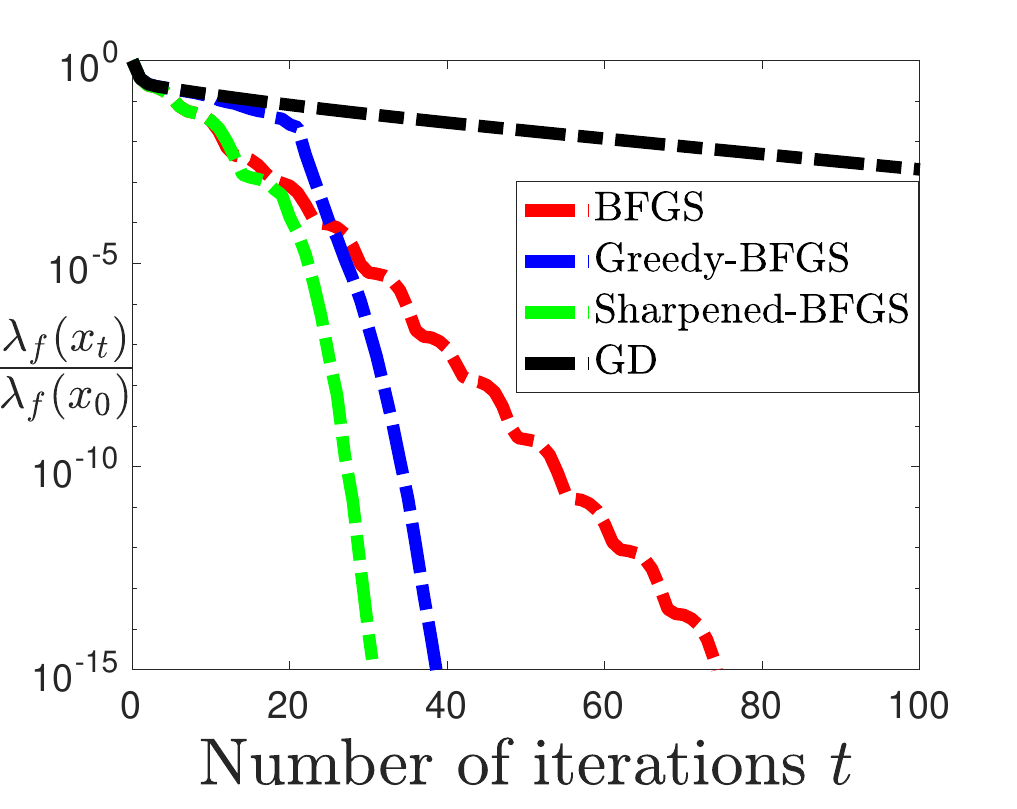}
    \caption{[Dataset svmguide3.}
\end{subfigure}
\hfill
\begin{subfigure}{0.32\textwidth}
    \includegraphics[width=\textwidth]{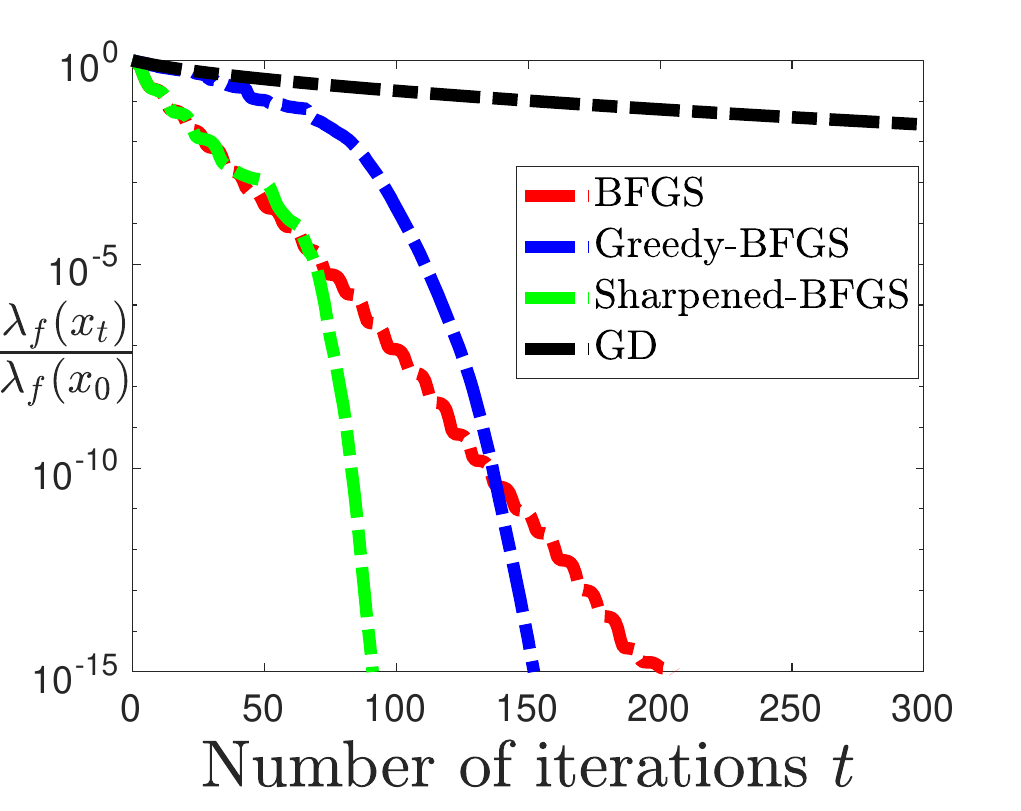}
    \caption{Dataset phishing.}
\end{subfigure}
\hfill
\begin{subfigure}{0.32\textwidth}
    \includegraphics[width=\textwidth]{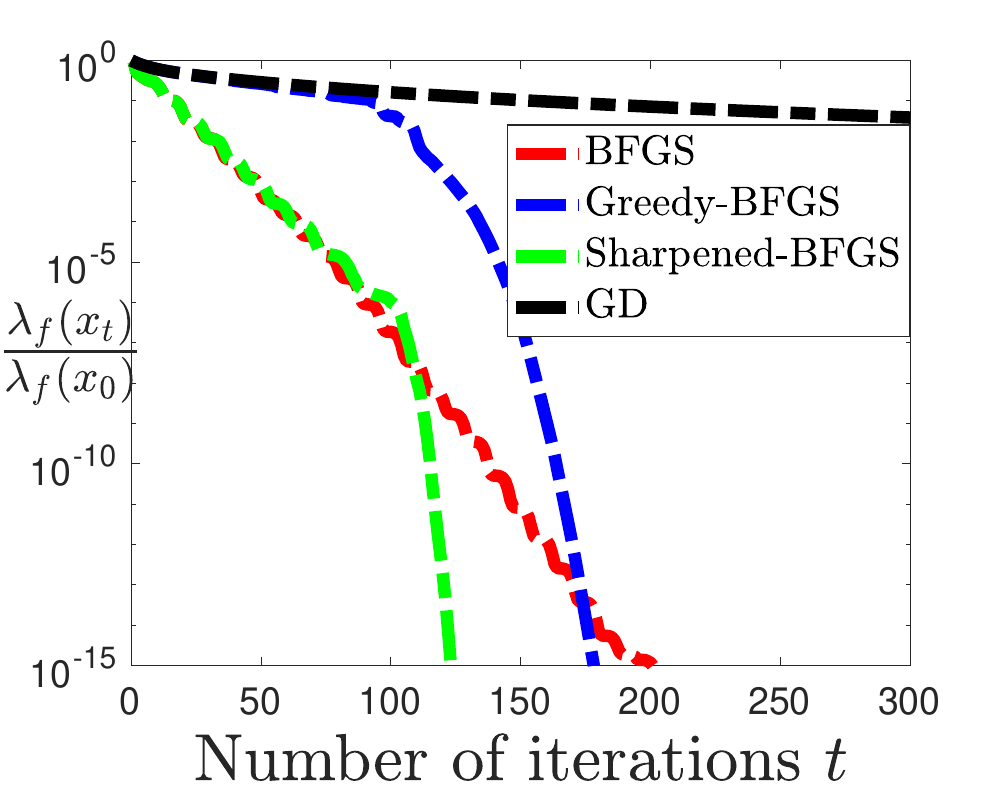}
    \caption{Dataset mushrooms.}
\end{subfigure}
\begin{subfigure}{0.32\textwidth}
    \includegraphics[width=\textwidth]{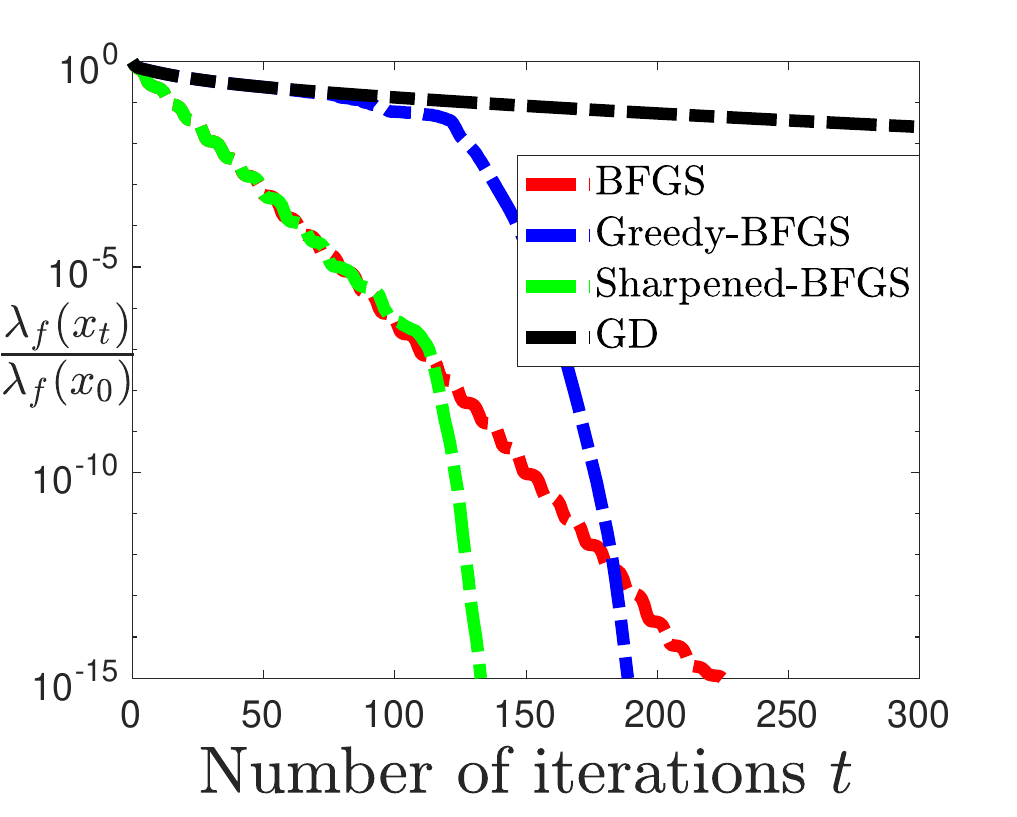}
    \caption{Dataset a9a.}
\end{subfigure}
\hfill
\begin{subfigure}{0.32\textwidth}
    \includegraphics[width=\textwidth]{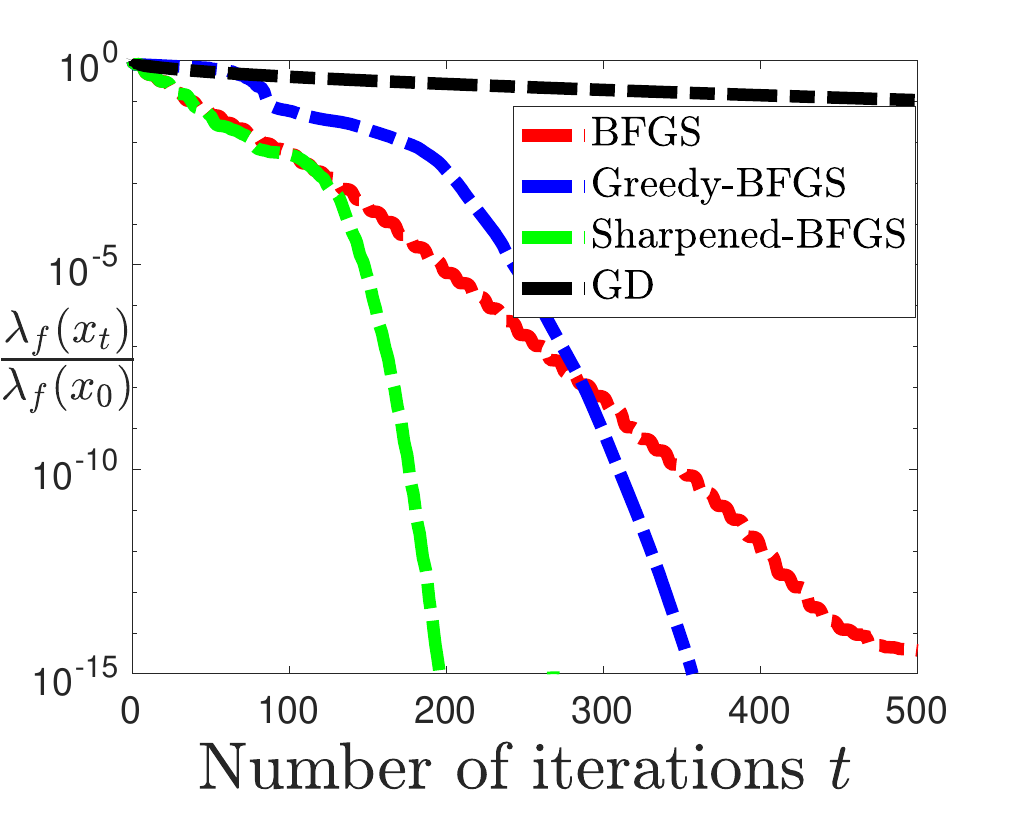}
    \caption{Dataset connect-4.}
\end{subfigure}
\hfill
\begin{subfigure}{0.32\textwidth}
    \includegraphics[width=\textwidth]{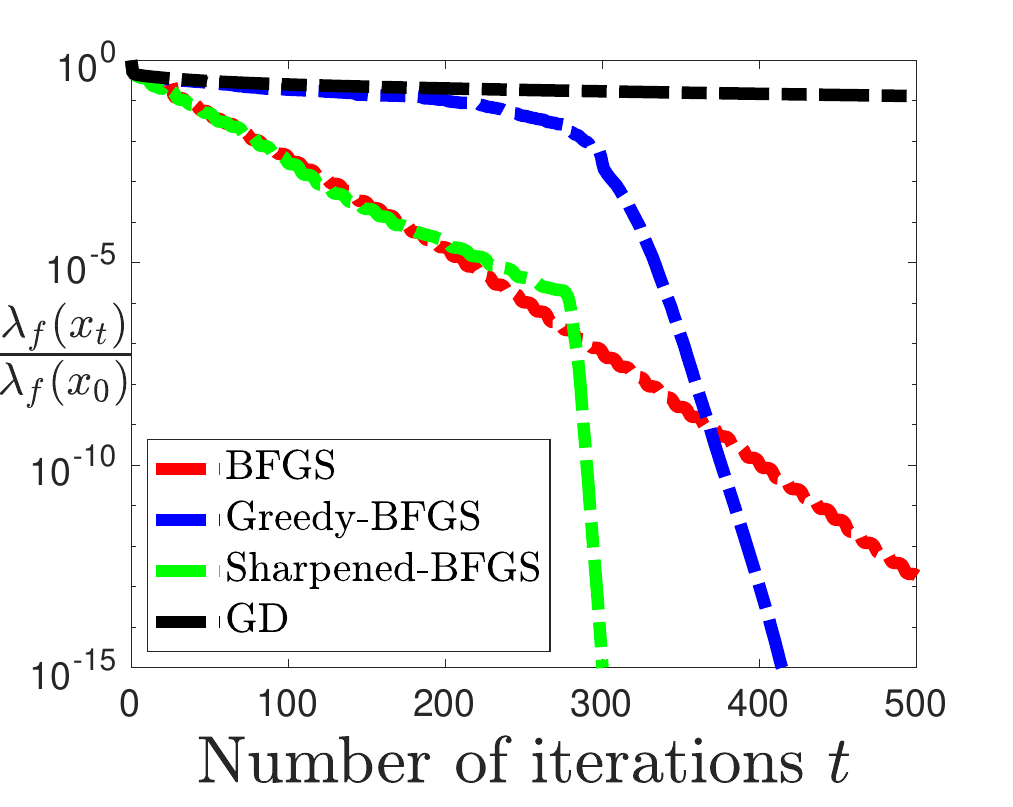}
    \caption{Dataset w8a.}
\end{subfigure}
\begin{subfigure}{0.32\textwidth}
    \includegraphics[width=\textwidth]{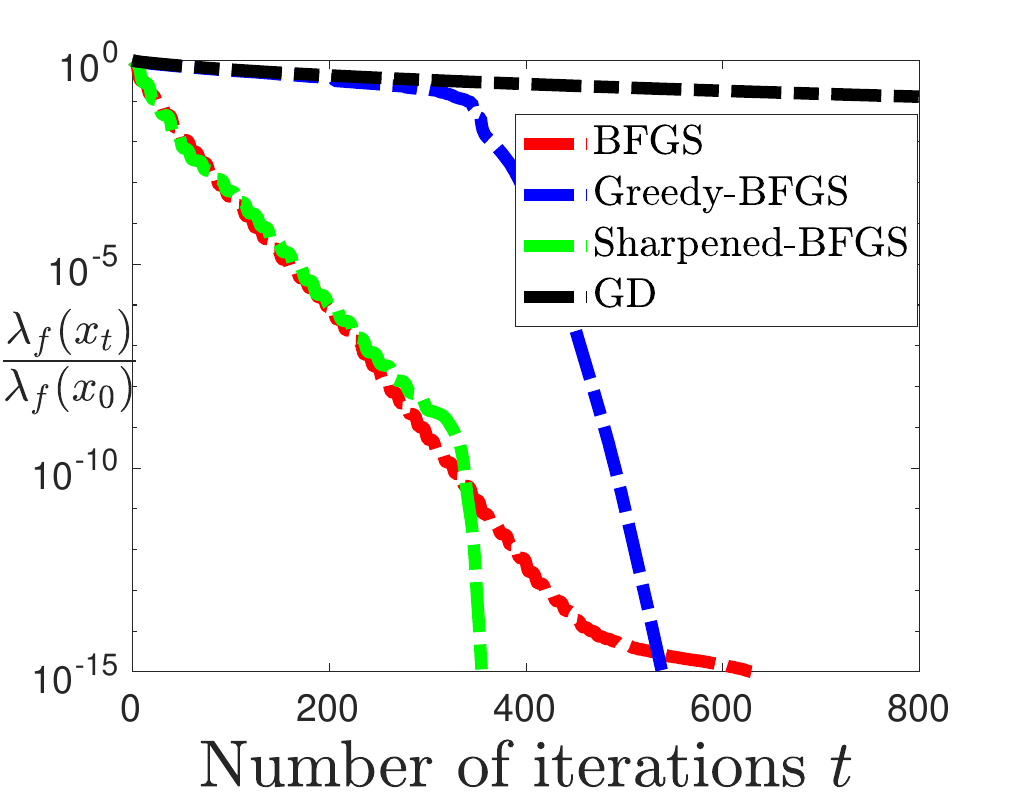}
    \caption{Dataset protein.}
\end{subfigure}
\hfill
\begin{subfigure}{0.32\textwidth}
    \includegraphics[width=\textwidth]{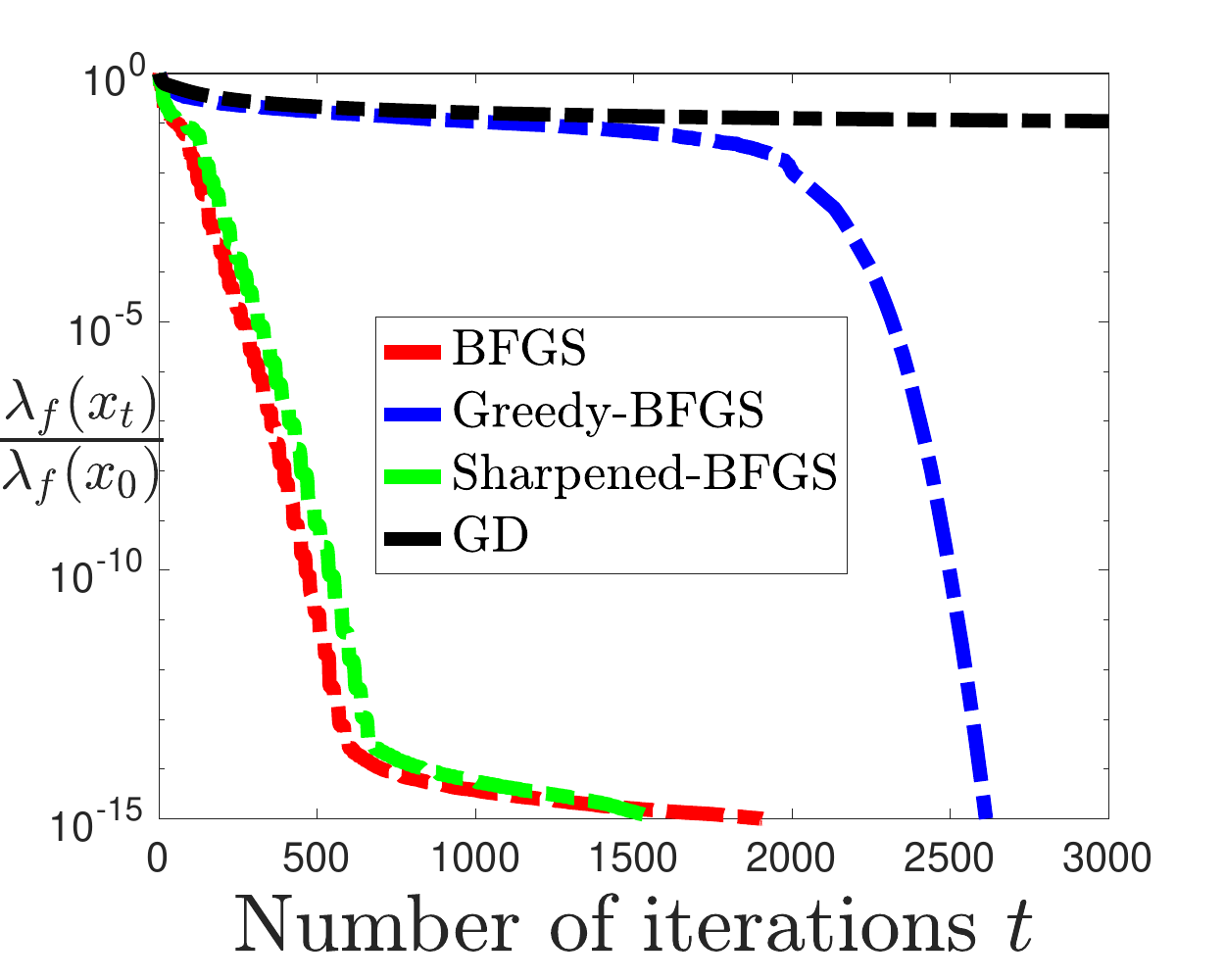}
    \caption{Dataset colon-cancer.}
\end{subfigure}
\hfill
\begin{subfigure}{0.32\textwidth}
    \includegraphics[width=\textwidth]{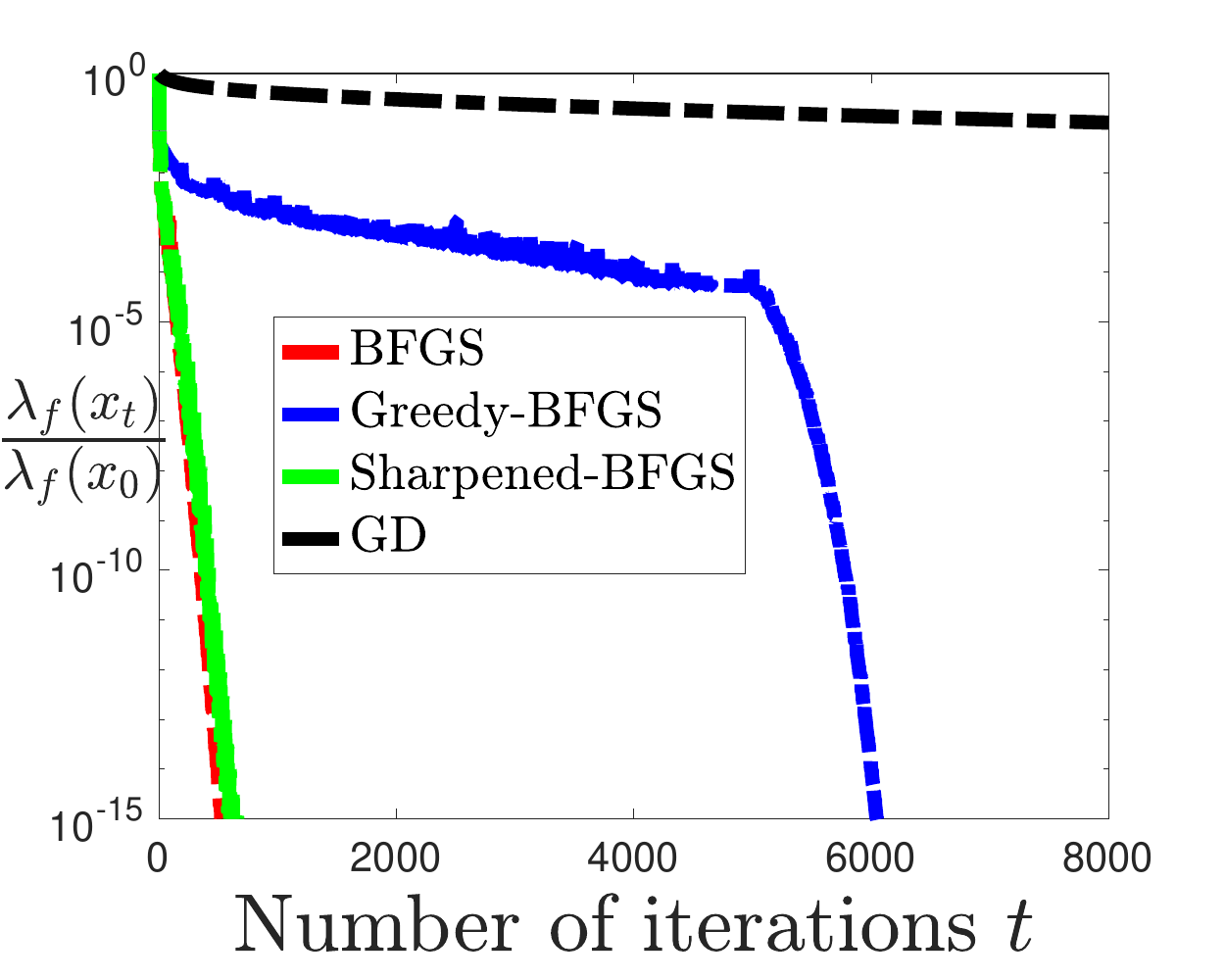}
    \caption{Dataset gisette.}
\end{subfigure}
\caption{Comparison of BFGS, Greedy-BFGS, Sharpened-BFGS and gradient descent (GD) on different datasets.}
\label{all_plots}
\end{figure}

In this section, we present our numerical experiments on different datasets to compare the performance of Sharpened-BFGS with BFGS and Greedy-BFGS. We focus on the following logistic regression problem with $l_2$ regularization
\begin{equation}\label{eq_numerical_experiment}
    \min_{x \in \mathbb{R}^d} f(x) = \frac{1}{N}\sum_{i = 1}^{N}\ln{(1 + e^{-y_i z_i^\top x})} + \frac{\mu}{2}\|x\|^2,
\end{equation}
where $\{z_i\}_{i = 1}^{N}$ are the data points and $\{y_i\}_{i = 1}^{N}$ are their corresponding labels. We assume that $z_i \in \mathbb{R}^d$ and $y_i \in \{-1, 1\}$ for all $1 \leq i \leq N$. This objective function is strongly convex with parameter $\mu > 0$. We normalize all data points such that $\|z_i\| = 1$ for all $1 \leq i \leq N$. Hence, the gradient of the function $f(x)$ is smooth with parameter $L = 1/4 + \mu$. The logistic regression objective function is also strongly self-concordant; check Section 5.1 in \cite{rodomanov2020greedy}. Therefore, the objective function $f(x)$ defined in \eqref{eq_numerical_experiment} satisfies Assumptions~\ref{ass_str_cvx_smooth}-\ref{ass_str_concordant}.

\begin{table}[t]
  \centering
  \begin{tabular}{ |c|c|c|c| }
    \hline
    Dataset & $N$ & $d$ & $\mu$ \\
    \hline
    \hline
    svmguide3 & 1243 & 21 & 0.01 \\
    \hline
    phishing & 11055 & 68 & 0.001 \\
    \hline
    mushrooms & 8124 & 112 & 0.001 \\
    \hline
    a9a & 32561 & 123 & 0.001 \\
    \hline
    connect-4 & 67557 & 126 & 0.0001 \\
    \hline
    w8a & 49749& 300 & 0.0001 \\
    \hline
    protein & 17766 & 357 & 0.0001 \\
    \hline
    colon-cancer & 62 & 2000 & 0.00001 \\
    \hline
    gisette & 6000 & 5000 & 0.00001 \\
    \hline
  \end{tabular}
  \caption{Sample size~$N$, dimension~$d$, regularization parameter~$\mu$.}
  \label{tab_2}
\end{table}

We conduct our experiments on eight datasets. All the parameters (sample size $N$, dimension $d$ and regularization parameter $\mu$) of these different datasets are summarized in Table~\ref{tab_2}. The regularization parameter $\mu$ is chosen from the set $
\mathcal{A} = \{10^{-5}, 10^{-4}, 10^{-3}, 10^{-2}, 10^{-1}, 1, 10\}$
to achieve the best performance. The algorithms that we study are (i) Sharpened-BFGS, (ii) standard BFGS, (iii) Greedy-BFGS, and (iv) gradient descent (GD). We initialize all the algorithms with the same initial point $x_0 = (1/d^{3/2})*\vec{\mathbf{1}}$, where $\vec{\mathbf{1}} \in \mathbb{R}^d$ is the one vector.  We set the initial Hessian approximation matrix as $LI$ and set the stepsize to 1 for all QN methods.  The step size of gradient descent is set as $1/L$ to achieve its linear convergence rate on each dataset. In practice, we found it is better not to apply the correction strategy in the hybrid and greedy methods (i.e., simply set $\hat{G}_t = \bar{G}_t$ in step $7$ of Algorithm~\ref{algo_general}). The convergence rates of the ratio ${\lambda_{f}(x_t)}/{\lambda_{f}(x_0)}$ versus the number of iterations $t$ are presented in Figure~\ref{all_plots}.

We observe that Sharpened-BFGS outperforms both the classical and greedy 
BFGS methods. More specifically, in the initial phase of the convergence process, Sharpened-BFGS exploits the Newton direction approximation of BFGS and has a fast convergence like BFGS, while Greedy-BFGS has a slower convergence at the beginning as it Hessian approximation is not accurate yet. Then, once the time index increases and almost approaches $d$, and $d$ greedy updates are accomplished, the Hessian approximation of Greedy-BFGS becomes more accurate. As a result, Greedy-BFGS achieves a very fast convergence rate at this turning point. Similarly, Sharpened-BFGS follows the same fast convergence of Greedy-BFGS almost at the same time index, as it also exploits the Hessian approximation update rule in Greedy-BFGS. This behavior is consistent over all the considered datasets in our experiments, as illustrated in Figure~\ref{all_plots}. We should also add that 
these empirical observations are consistent with our theoretical findings and the performance comparisons of these algorithms in Section~\ref{sec:discussion}. 

\vspace{-8mm}

\section{Conclusions}\label{sec:conclusion}

\vspace{-5mm}

In this paper, we proposed a novel quasi-Newton method called Sharpened-BFGS for solving unconstrained convex optimization problems, where the objective function is strongly convex with $\mu$, its gradient is smooth with $L$, and it is strongly self-concordant with $M$. Sharpened-BFGS benefits from the Newton direction approximation of BFGS as well as Hessian approximation of Greedy-BFGS. Using these properties, we proved that the proposed Sharpened-BFGS achieves a superlinear convergence rate of $\mathcal{O}((1 - \frac{\mu}{dL})^{\frac{t(t - 1)}{4}} (\frac{dL}{t\mu})^{\frac{t}{2}})$, which is faster than quadratic rate. We also compared the convergence results of our method with the classical BFGS and Greedy-BFGS methods and highlighted how Sharpened-BFGS takes advantage of the Newton direction approximation in BFGS and the Hessian approximation in Greedy-BFGS. We also numerically illustrated the advantages of our proposed method  against BFGS and Greedy-BFGS.

\appendix
\section*{Appendix}

\section{Preliminary Lemmas}
In this subsection, we develop some technical preliminaries which are critical in the path towards establishing our main convergence results. We begin with the following lemma regarding the BFGS operator defined in Sec. \ref{sec:standard_BFGS} 
\begin{lemma}\label{lemma_1}
Consider positive definite matrices $A, G \in \mathbb{R}^{d \times d}$ and suppose that $G_{+} = BFGS(A, G, u)$ as defined in \eqref{BFGS_update} for any $u \in \mathbb{R}^d\backslash\{0\}$. Then, the following results hold:

\begin{enumerate}

    \item
    For any constants $\xi, \eta \geq 1$, we have
    \begin{equation}\label{lemma_1_1}
        \frac{1}{\xi}A \preceq G \preceq \eta A \quad \Rightarrow \quad \frac{1}{\xi}A \preceq G_{+} \preceq \eta A.
    \end{equation}
    
    \item
    If $A \preceq G$, then we have
    \begin{equation}\label{lemma_1_2}
        \sigma(A, G) - \sigma(A, G_+) \geq \theta^2(A,G,u).
    \end{equation}

    \item
    If $\frac{1}{\xi}A \preceq G$  and $ \theta(A, G, u) \leq \xi$, where $\xi \geq 1$ is a constant, then
    \begin{equation}\label{lemma_1_3}
        \sigma(A, G) - \sigma(A, G_+) \geq \frac{1}{4\xi^2}\theta^2(A, G, u) - \ln{\xi}.
    \end{equation}
    
\end{enumerate}

\end{lemma}

\begin{proof}

Check Lemma 2.1 in \cite{rodomanov2020rates} for the proof of \eqref{lemma_1_1} and check Lemma 2.2 in \cite{rodomanov2020rates} for the proof of \eqref{lemma_1_2}. Now we prove \eqref{lemma_1_3}. We denote $Det(A)$ as the determinant of the matrix $A \in \mathbb{R}^{d \times d}$. Applying results of Lemma 2.4 in \cite{rodomanov2020rates}, we obtain that
\begin{equation}\label{proof_lemma_1_1}
    \psi(A, G) - \psi(A, G_{+}) \geq \omega(\frac{\theta(A, G, u)}{\xi}),
\end{equation}
where
\begin{equation}
\psi(A, G) := Tr(A^{-1}(G - A)) - \ln{Det(A^{-1}G)} = \sigma(A, G) - \ln{Det(A^{-1}G)},
\end{equation}
and
\begin{equation}
\omega(t) := t - \ln(1 + t), \quad \forall t \geq -1.
\end{equation}
Thus, we obtain that
\begin{equation}
\begin{split}
    \psi(A, G) - \psi(A, G_{+}) & =  \sigma(A, G) - \sigma(A, G_{+}) - \ln{Det(A^{-1}G)} + \ln{Det(A^{-1}G_{+})}\\
    & = \sigma(A, G) -  \sigma(A, G_{+}) + \ln{Det(G^{-1}G_{+})}.
\end{split}
\end{equation}
From Lemma 6.2 of \cite{rodomanov2020rates}, we have that
\begin{equation}
    Det(G^{-1}G_{+}) = \frac{u^\top Au}{u^\top Gu}.
\end{equation}
Hence, we get that
\begin{equation}\label{proof_lemma_1_2}
    \psi(A, G) - \psi(A, G_{+})  =  \sigma(A, G) - \sigma(A, G_{+}) + \ln{\frac{u^\top Au}{u^\top Gu}}.
\end{equation}
Substituting \eqref{proof_lemma_1_2} into the \eqref{proof_lemma_1_1}, we obtain that
\begin{equation}\label{proof_lemma_1_3}
\sigma(A, G) - \sigma(A, G_{+}) \geq \omega(\frac{\theta(A, G, u)}{\xi}) - \ln{\frac{u^\top Au}{u^\top Gu}}.
\end{equation}
Notice that the function $\omega(t)$ satisfies the following property,
\begin{equation}
\omega(t) \geq \frac{t^2}{2(t + 1)}, \quad \forall t \geq 0.
\end{equation}
Thus, we derive that
\begin{equation}\label{proof_lemma_1_4}
\omega(\frac{\theta(A, G, u)}{\xi}) \geq \frac{\theta^2(A, G, u)/\xi^2}{2(\theta(A, G, u)/\xi + 1)} = \frac{\theta^2(A, G, u)}{2\xi(\theta(A, G, u) + \xi)} \geq \frac{\theta^2(A, G, u)}{4\xi^2},
\end{equation}
where the second inequality is due to the condition $\theta(A, G, u) \leq \xi$. From condition $\frac{1}{\xi}A \preceq G$, we know that for any $u \in \mathbb{R}^d\backslash\{0\}$
\begin{equation}\label{proof_lemma_1_5}
\frac{u^\top Au}{u^\top Gu} \leq \xi.
\end{equation}
Substituting \eqref{proof_lemma_1_4} and \eqref{proof_lemma_1_5} into \eqref{proof_lemma_1_3}, we achieve conclusion \eqref{lemma_1_3}.

\end{proof}

In the following lemma, we show that the Hessians of the strongly self-concordant function at two different points.

\begin{lemma}\label{lemma_2}
Suppose the objective function $f(x)$ is strongly self-concordant with constant $M > 0$. Consider $x, y \in \mathbb{R}^d$, $r = \|y - x\|_x$ and $J = \int_{0}^{1}\nabla^2{f(x + \tau(y - x))}d\tau$. Then, we have that
\begin{equation}\label{lemma_2_1}
    \frac{\nabla^2{f(x)}}{1 + Mr} \preceq \nabla^2{f(y)} \preceq (1 + Mr)\nabla^2{f(x)}.
\end{equation}
\begin{equation}\label{lemma_2_2}
    \frac{\nabla^2{f(x)}}{1 + \frac{Mr}{2}} \preceq J \preceq (1 + \frac{Mr}{2})\nabla^2{f(x)}.
\end{equation}
\begin{equation}\label{lemma_2_3}
    \frac{\nabla^2{f(y)}}{1 + \frac{Mr}{2}} \preceq J \preceq (1 + \frac{Mr}{2})\nabla^2{f(y)}.
\end{equation}
\end{lemma}

\begin{proof}
Check Lemma 4.2 in \cite{rodomanov2020greedy}.
\end{proof}

\begin{lemma}\label{lemma_3}
Suppose the objective function $f(x)$ satisfies the Assumption~\ref{ass_str_cvx_smooth} and \ref{ass_str_concordant}. Consider the following update
\begin{equation}\label{lemma_3_1}
x_{t + 1} = x_t - G_t^{-1}\nabla{f(x_t)},
\end{equation}
where $G_{t} \in \mathbb{R}^{d \times d}$ is the s.p.d. Hessian approximation matrix satisfying that
\begin{equation}\label{lemma_3_2}
\nabla^2{f(x_t)} \preceq G_t \preceq \eta \nabla^2{f(x_t)},
\end{equation}
where $\eta \geq 1$ is some constant. Suppose the following condition holds
\begin{equation}\label{lemma_3_3}
    M\lambda_{t} \leq 2.
\end{equation}
Denote that $r_t = \|x_{t + 1} - x_t\|_{x_t}$ and $J_t = \int_{0}^{1}\nabla^2{f(x_t + \tau(x_{t + 1} - x_{t}))d\tau}$. Then, we have
\begin{equation}\label{lemma_3_4}
    r_t \leq \lambda_t,
\end{equation}
\begin{equation}\label{lemma_3_5}
    \theta(J_t, G_t, x_{t + 1} - x_{t}) \leq \frac{\eta - 1 + \frac{M\lambda_{t}}{2}}{\eta}.
\end{equation}
\end{lemma}

\begin{proof}
From \eqref{lemma_3_1}, we have that
\begin{equation}
\begin{split}
    r_t & = \|x_{t + 1} - x_t\|_{x_t} = \left(\nabla{f(x_t)}^\top G^{-1}_t\nabla^2{f(x_t)}G^{-1}_t\nabla{f(x_t)}\right)^\frac{1}{2}\\
    & \leq  \left(\nabla{f(x_t)}^\top G^{-1}_t\nabla{f(x_t)}\right)^\frac{1}{2} \leq  \left(\nabla{f(x_t)}^\top \nabla^2{f(x_t)}^{-1}\nabla{f(x_t)}\right)^\frac{1}{2} = \lambda_t,
\end{split}
\end{equation}
where the inequalities hold due to \eqref{lemma_3_2}. Therefore, \eqref{lemma_3_4} holds. Now we condition \eqref{lemma_3_5}. Using \eqref{lemma_3_2} and \eqref{lemma_2_2} of Lemma~\ref{lemma_2}, we obtain that
\begin{equation}
    \frac{1}{1 + \frac{Mr_t}{2}}J_t \preceq \nabla^2{f(x_t)} \preceq G_t \preceq \eta \nabla^2{f(x_t)} \preceq \eta(1 + \frac{Mr_t}{2})J_t.
\end{equation}
Using $r_t \leq \lambda_t$ from \eqref{lemma_3_4}, we get that
\begin{equation}
    \frac{1}{1 + \frac{M\lambda_t}{2}}J_t \preceq G_t \preceq \eta(1 + \frac{M\lambda_t}{2})J_t.
\end{equation}
Hence, we have
\begin{equation}
    -\left(1 - \frac{1}{\eta(1 + \frac{M\lambda_t}{2})}\right)J^{-1}_t \preceq G^{-1}_t - J^{-1}_t \preceq \frac{M\lambda_t}{2}J^{-1}_t.
\end{equation}
Notice that
\begin{equation}
    \left(1 - \frac{1}{\eta(1 + \frac{M\lambda_t}{2})}\right) \leq 1 - \frac{1 - \frac{M\lambda_t}{2}}{\eta} = \frac{\eta - 1 + \frac{M\lambda_t}{2}}{\eta}.
\end{equation}
Since $M\lambda_t \leq 2$ and $\eta \geq 1$, we have
\begin{equation}
    \frac{M\lambda_t}{2} = 1 - (1 - \frac{M\lambda_t}{2}) \leq 1 - \frac{1 - \frac{M\lambda_t}{2}}{\eta} = \frac{\eta - 1 + \frac{M\lambda_t}{2}}{\eta}.
\end{equation}
Therefore, we have
\begin{equation}
    -\frac{\eta - 1 + \frac{M\lambda_t}{2}}{\eta}J^{-1}_t \preceq G^{-1}_t - J^{-1}_t \preceq \frac{\eta - 1 + \frac{M\lambda_t}{2}}{\eta}J^{-1}_t.
\end{equation}
Hence, we get
\begin{equation}
    (G^{-1}_t - J^{-1}_t)J_t(G^{-1}_t - J^{-1}_t) \preceq \left(\frac{\eta - 1 + \frac{M\lambda_t}{2}}{\eta}\right)^2 J^{-1}_t,
\end{equation}
\begin{equation}
    s_t^\top G_t(G^{-1}_t - J^{-1}_t)J_t(G^{-1}_t - J^{-1}_t)G_t s_t \leq \left(\frac{\eta - 1 + \frac{M\lambda_t}{2}}{\eta}\right)^2 s_t^\top G_tJ ^{-1}_t G_t s_t,
\end{equation}
where $s_t = x_{t + 1} - x_t$ is the variable difference. Therefore, by the definition of $\theta$ in \eqref{theta}, we prove conclusion \eqref{lemma_3_5},
\begin{equation}
\begin{split}
    \theta(J_t, G_t, x_{t + 1} - x_{t}) & = \left(\frac{s_t^\top (G_t - J_t) J_t^{-1} (G_t - J_t) s_t}{s_t^\top G_t J_t^{-1} G_t s_t}\right)^{\frac{1}{2}}\\
    & = \left(\frac{s_t^\top G_t(J_t^{-1} - G_t^{-1})J_t(J_t^{-1} - G_t^{-1}) G_ts_t}{s_t^\top G_t J_t^{-1} G_t s_t}\right)^{\frac{1}{2}} \leq \frac{\eta - 1 + \frac{M\lambda_t}{2}}{\eta}.
\end{split}
\end{equation}

\end{proof}

\section{Proof of Theorem~\ref{thm_quadratic_linear}}\label{sec:proof_of_thm_quadratic_linear}

First, we use induction to prove the following condition

\begin{equation}\label{proof_thm_quadratic_linear_1}
A \preceq G_t \preceq \frac{L}{\mu}A, \quad \forall t \geq 0.
\end{equation}

From $\mu I \preceq A \preceq LI$, we observe that the initial Hessian approximation matrix $G_0 = LI$ satisfies $A \preceq G_0 \preceq \frac{L}{\mu}A$. Hence, condition \eqref{proof_thm_quadratic_linear_1} holds for $t = 0$. We assume that condition \eqref{proof_thm_quadratic_linear_1} holds for $t = k$, i.e., $A \preceq G_k \preceq \frac{L}{\mu}A$, where $k \geq 0$. Applying \eqref{lemma_1_1} of Lemma~\ref{lemma_1} to the update in step $4$ of Algorithm~\ref{algo_quadratic}, we obtain that $A \preceq \bar{G}_k \preceq \frac{L}{\mu}A$. Applying \eqref{lemma_1_1} of Lemma~\ref{lemma_1} again to the update in step $6$ of Algorithm~\ref{algo_quadratic}, we obtain that $A \preceq G_{k+1} \preceq \frac{L}{\mu}A$. Therefore, condition \eqref{proof_thm_quadratic_linear_1} holds for $t = k + 1$. By induction, we prove that condition \eqref{proof_thm_quadratic_linear_1} holds for any $t \geq 0$. Moreover, this condition implies that for any $t \geq 0$, we have
\begin{equation}
    0 \preceq A^{-1} - G_{t}^{-1} \preceq (1 - \frac{\mu}{L})A^{-1}.
\end{equation}
Hence, we obtain that
\begin{equation}
(G_t - A)A^{-1}(G_t - A) = G_t(A^{-1} - G_t^{-1})A(A^{-1} - G_t^{-1})G_t \preceq (1 - \frac{\mu}{L})^2G_t A^{-1}G_t,
\end{equation}
\begin{equation}
s_t^\top (G_t - A)A^{-1}(G_t - A)s_t \leq (1 - \frac{\mu}{L})^2 s_t G_t A^{-1}G_t s_t,
\end{equation}
where $s_t = x_{t + 1} - x_t$ is the variable difference. By the definition of $\theta$ in \eqref{theta}, we have that
\begin{equation}
\theta(A, G_t, x_{t + 1} - x_{t}) = \left(\frac{s_t^\top (G_t - A)A^{-1}(G_t - A)s_t}{s_t G_t A^{-1}G_t s_t}\right)^{\frac{1}{2}} \leq 1 - \frac{\mu}{L}.
\end{equation}
Therefore, \eqref{thm_quadratic_linear_1} holds for any $t \geq 0$. Applying \eqref{lemma_quadratic_theta_1} of Lemma~\ref{lemma_quadratic_theta}, we prove that
\begin{equation}
 \lambda_{t + 1} = \theta(A, G_t, x_{t + 1} - x_{t})\lambda_{t} \leq (1 - \frac{\mu}{L}) \lambda_{t}, \qquad \forall t \geq 0.
\end{equation}
Hence, we prove the linear convergence rate of \eqref{thm_quadratic_linear_2}. \hfill \qed
\section{Proof of Lemma~\ref{lemma_quadratic_potential}}\label{sec:proof_of_lemma_quadratic_potential}

The initial Hessian approximation matrix $G_0 = LI \succeq A$. Applying the same induction technique used in the proof of Theorem~\ref{thm_quadratic_linear}, we can prove that for any $t \geq 0$
\begin{equation}
G_t \succeq A, \qquad  \bar{G_t} \succeq A,
\end{equation}
where $\bar{G}_t$ is defined in step $4$ of Algorithm~\ref{algo_quadratic}. Using \eqref{lemma_1_2} of Lemma~\ref{lemma_1}, we have that
\begin{equation}
    \sigma(A, G_t) - \sigma(A, \bar{G_t}) \geq \theta^2(A, G_t , x_{t + 1} - x_t), \quad \forall t \geq 0.
\end{equation}
Applying \eqref{lemma_BFGS_greedy_1} of Lemma~\ref{lemma_BFGS_greedy} to the step $6$ of Algorithm~\ref{algo_quadratic}, we obtain that
\begin{equation}
    \sigma(A, G_{t + 1}) \leq (1 - \frac{\mu}{dL})\sigma(A, \bar{G_t}), \quad \forall t \geq 0.
\end{equation}
We prove conclusion \eqref{lemma_quadratic_potential_1} by combining and regrouping the above two inequalities. Now we prove condition \eqref{lemma_quadratic_potential_2}. Recall and define the following shorthanded notations
\begin{equation}
\theta_t = \theta(A, G_t, x_{t + 1} - x_{t}), \qquad \sigma_t = \sigma(A, G_t), \qquad c = \frac{\mu}{dL}.
\end{equation}
Condition \eqref{lemma_quadratic_potential_1} is equivalent to
\begin{equation}
    \sigma_{t} \leq (1 - c)\sigma_{t - 1} - (1 - c)\theta^2_{t - 1}, \quad \forall t \geq 1.
\end{equation}
Applying the above inequality recursively, we can derive that
\begin{equation}
\begin{split}
    \sigma_{t} & \leq (1 - c)\sigma_{t - 1} - (1 - c)\theta^2_{t - 1}\\
    & \leq (1 - c)^2\sigma_{t - 2} - (1 - c)^2\theta^2_{t - 2} - (1 - c)\theta^2_{t - 1}\\
    & \leq (1 - c)^t\sigma_{0} - \sum_{i = 0}^{t - 1}(1 - c)^{t - i}\theta^2_i.
\end{split}
\end{equation}
The above inequality indicates that
\begin{equation}\label{proof_lemma_quadratic_potential_1}
\sum_{i = 0}^{t - 1}(1 - c)^{t - i}\theta^2_i \leq (1 - c)^t\sigma_{0} - \sigma_{t} \leq (1 - c)^t\sigma_{0}.
\end{equation}
Dividing the term $(1 - c)^t$ on both sides of the above inequality, we can obtain that
\begin{equation}
\sum_{i = 0}^{t - 1}\frac{\theta^2_i}{(1 - c)^i} \leq \sigma_{0}.
\end{equation}
Hence, we prove the result \eqref{lemma_quadratic_potential_2} since $c = \frac{\mu}{dL}$. \hfill \qed

\section{Proof of Theorem~\ref{thm_quadratic_superlinear}}\label{sec:proof_of_thm_quadratic_superlinear}

Using the condition $A^{-1} \preceq \frac{1}{\mu}I$ and recalling the notation $c = \frac{\mu}{dL}$, we can upper bound $\sigma_0$ by
\begin{equation}
\sigma_0 = \sigma(A, G_0) = Tr(A^{-1}LI) - d \leq Tr(\frac{L}{\mu}I)- d = d(\frac{L}{\mu} - 1) \leq d\frac{L}{\mu} = \frac{1}{c}.
\end{equation}
Combining the above upper bound and \eqref{lemma_quadratic_potential_2}, we derive that
\begin{equation}\label{proof_thm_quadratic_superlinear_1}
\sum_{i = 0}^{t - 1}\frac{\theta^2_i}{(1 - c)^i} \leq \sigma_{0} \leq \frac{1}{c}.
\end{equation}
From \eqref{lemma_quadratic_theta_1} of Lemma~\ref{lemma_quadratic_theta}, we obtain that
\begin{equation}\label{proof_thm_quadratic_superlinear_2}
    \frac{\lambda_{t}}{\lambda_{0}} = \prod_{i = 0}^{t - 1}\frac{\lambda_{i + 1}}{\lambda_{i}} = \prod_{i = 0}^{t - 1}\theta_i = \prod_{i = 0}^{t - 1}(1 - c)^{\frac{i}{2}}\frac{\theta_i}{(1 - c)^{\frac{i}{2}}} = \prod_{i = 0}^{t - 1}(1 - c)^{\frac{i}{2}}\prod_{i = 0}^{t - 1}\frac{\theta_i}{(1 - c)^{\frac{i}{2}}} = (1 - c)^{\frac{t(t - 1)}{4}}\prod_{i = 0}^{t - 1}\frac{\theta_i}{(1 - c)^{\frac{i}{2}}}.
\end{equation}
Using the arithmetic-geometric mean inequality and \eqref{proof_thm_quadratic_superlinear_1}, we derive that
\begin{equation}\label{proof_thm_quadratic_superlinear_3}
    \prod_{i = 0}^{t - 1}\frac{\theta_i}{(1 - c)^{\frac{i}{2}}} = \left[\prod_{i = 0}^{t - 1}\frac{\theta^2_i}{(1 - c)^i}\right]^{\frac{1}{2}} \leq \left[\frac{1}{t}\sum_{i = 0}^{t - 1}\frac{\theta^2_i}{(1 - c)^i}\right]^{\frac{t}{2}} \leq \left(\frac{1}{ct}\right)^{\frac{t}{2}}.
\end{equation}
Leveraging \eqref{proof_thm_quadratic_superlinear_2} and \eqref{proof_thm_quadratic_superlinear_3}, we achieve the final convergence rate of \eqref{thm_quadratic_superlinear_1}
\begin{equation}
\frac{\lambda_{t}}{\lambda_{0}} \leq (1 - c)^{\frac{t(t - 1)}{4}}\left(\frac{1}{ct}\right)^{\frac{t}{2}} = (1 - \frac{\mu}{dL})^{\frac{t(t - 1)}{4}} (\frac{dL}{t\mu})^{\frac{t}{2}}, \qquad \forall t \geq 1.
\end{equation}\hfill \qed

\section{Proof of Theorem~\ref{thm_general_linear}}\label{sec:proof_of_thm_general_linear}

First, we use induction to prove the following condition

\begin{equation}\label{proof_thm_general_linear_1}
\nabla^2{f(x_t)} \preceq G_t \preceq \xi_t \frac{L}{\mu}\nabla^2{f(x_t)}, \quad \forall t \geq 0,
\end{equation}
where
\begin{equation}\label{proof_thm_general_linear_2}
    \xi_0 = 1 \quad \text{and} \quad \xi_t = e^{2M\sum_{i = 0}^{t - 1}r_i}, \quad \forall t \geq 1.
\end{equation}

We use induction to prove \eqref{proof_thm_general_linear_1} and \eqref{proof_thm_general_linear_2}. When $t = 0$, from Assumption~\ref{ass_str_cvx_smooth} we know that
\begin{equation}
\nabla^2{f(x_0)} \preceq G_0 = LI \preceq \frac{L}{\mu}\nabla^2{f(x_0)}.
\end{equation}
Hence, \eqref{proof_thm_general_linear_1} and \eqref{proof_thm_general_linear_2} hold for $t = 0$. Suppose that \eqref{proof_thm_general_linear_1} and \eqref{proof_thm_general_linear_2} hold for $t = k$, we have that
\begin{equation}\label{proof_thm_general_linear_3}
    \nabla^2{f(x_k)} \preceq G_k \preceq \xi_k \frac{L}{\mu}\nabla^2{f(x_k)},\quad \xi_k = e^{2M\sum_{i = 0}^{k - 1}r_i}.
\end{equation}
Now we consider the case of $t = k + 1$. Condition \eqref{lemma_2_2} of Lemma~\ref{lemma_2} indicates that
\begin{equation}
    \frac{1}{1 + \frac{Mr_k}{2}}J_k \preceq \nabla^2{f(x_k)} \preceq G_k \preceq \xi_k \frac{L}{\mu}\nabla^2{f(x_k)} \preceq \xi_k \frac{L}{\mu}(1 + \frac{Mr_k}{2})J_k.
\end{equation}
where $J_k = \int_{0}^{1}\nabla^2{f(x_k + \tau(x_{k + 1} - x_k))}d\tau$. Applying \eqref{lemma_1_1} of Lemma~\ref{lemma_1}, we have that
\begin{equation}
    \frac{1}{1 + \frac{Mr_k}{2}}J_k \preceq \bar{G}_k = BFGS(J_k, G_k, x_{k + 1} - x_k) \preceq \xi_k \frac{L}{\mu}(1 + \frac{Mr_k}{2})J_k,
\end{equation}
where the equality is due to step $5$ of Algorithm~\ref{algo_general}. Condition \eqref{lemma_2_3} of Lemma~\ref{lemma_2} indicates that
\begin{equation}
    \frac{1}{(1 + \frac{Mr_k}{2})^2}\nabla^2{f(x_{k + 1})} \preceq \frac{1}{1 + \frac{Mr_k}{2}}J_k \preceq \bar{G}_k \preceq \xi_k \frac{L}{\mu}(1 + \frac{Mr_k}{2})J_k, \preceq \xi_k \frac{L}{\mu}(1 + \frac{Mr_k}{2})^2\nabla^2{f(x_{k + 1})}.
\end{equation}
Multiplying the term $(1 + \frac{Mr_k}{2})^2$ on both sides of the above inequality, we get that
\begin{equation}\label{proof_thm_general_linear_4}
    \nabla^2{f(x_{k + 1})} \preceq (1 + \frac{Mr_k}{2})^2\bar{G}_k = \hat{G}_k \preceq \xi_k \frac{L}{\mu}(1 + \frac{Mr_k}{2})^4\nabla^2{f(x_{k + 1})},
\end{equation}
where the equality is due to step $7$ of Algorithm~\ref{algo_general}. Applying the fact $1 + x \leq e^x$, we have
\begin{equation}\label{proof_thm_general_linear_5}
\xi_k (1 + \frac{Mr_k}{2})^4 \leq \xi_k e^{2Mr_k} = e^{2M\sum_{i = 0}^{k - 1}r_i}e^{2Mr_k} = e^{2M\sum_{i = 0}^{k}r_i} = \xi_{k + 1},
\end{equation}
where the first equality is due to the induction assumption in \eqref{proof_thm_general_linear_3} and the last equality is due to the definition in \eqref{proof_thm_general_linear_2}. Substituting \eqref{proof_thm_general_linear_5} into \eqref{proof_thm_general_linear_4}, we have that
\begin{equation}
    \nabla^2{f(x_{k + 1})} \preceq \hat{G}_k \preceq \xi_{k + 1} \frac{L}{\mu}\nabla^2{f(x_{k + 1})}.
\end{equation}
Applying \eqref{lemma_1_1} of Lemma~\ref{lemma_1} again and step $9$ of Algorithm~\ref{algo_general}, we obtain that
\begin{equation}
    \nabla^2{f(x_{k + 1})} \preceq G_{k + 1} = BFGS(\nabla^2{f(x_{k + 1})}, \hat{G_k}, \bar{u}(\nabla^2{f(x_{k + 1})}, \hat{G_k})) \preceq \xi_{k + 1} \frac{L}{\mu}\nabla^2{f(x_{k + 1})}.
\end{equation}
Hence, \eqref{proof_thm_general_linear_1} and \eqref{proof_thm_general_linear_2} hold for $t = k + 1$. Therefore, We finish the proof of \eqref{proof_thm_general_linear_1} and \eqref{proof_thm_general_linear_2} using induction.

Now, we use induction again to prove the result of \eqref{thm_general_linear_2} and \eqref{thm_general_linear_3}. It's obvious that \eqref{thm_general_linear_3} holds for $t = 0$. Suppose that \eqref{thm_general_linear_3} holds for $0 \leq t \leq k$, we have that
\begin{equation}\label{proof_thm_general_linear_6}
M\lambda_t \leq M\lambda_0 \leq C_0\frac{\mu}{L} = \frac{\ln{\frac{3}{2}}}{4}\frac{\mu}{L} < 1 < 2, \qquad 0 \leq t \leq k,
\end{equation}
where we use the initial condition \eqref{thm_general_linear_1} and the fact $\mu \leq L$. Conditions \eqref{proof_thm_general_linear_1} and \eqref{proof_thm_general_linear_6} imply that \eqref{lemma_3_4} and \eqref{lemma_3_5} of Lemma~\ref{lemma_3} hold for all $0 \leq t \leq k$ where $\eta = \xi_t{L}/{\mu}$. Hence, we have that
\begin{equation}\label{proof_thm_general_linear_7}
    \theta(J_t, G_t, x_{t + 1} - x_{t}) \leq \frac{\eta - 1 + \frac{M\lambda_{t}}{2}}{\eta} = 1 - \frac{\mu}{L\xi_t}(1 - \frac{M\lambda_t}{2}), \qquad 0 \leq t \leq k.
\end{equation}
Applying the initial condition \eqref{thm_general_linear_1} and the induction assumption of \eqref{thm_general_linear_3} for $0 \leq t \leq k$, we observe that
\begin{equation}
M\sum_{i = 0}^{t}\lambda_i \leq M\lambda_0 \sum_{i = 0}^{t}(1 - \frac{\mu}{2L})^{i} \leq 2M\frac{L}{\mu}\lambda_0 \leq 2C_0, \qquad 0 \leq t \leq k.
\end{equation}
Consequently,
\begin{equation}\label{proof_thm_general_linear_8}
e^{2M\sum_{i = 0}^{t}\lambda_i} \leq e^{4C_0} = e^{\ln{\frac{3}{2}}} = \frac{3}{2}, \qquad 0 \leq t \leq k.
\end{equation}
Since $M\lambda_t < 1$ from \eqref{proof_thm_general_linear_6} and the fact that $1 - x/2 \geq e^{-x}$ for $x \in (0, 1)$, we get that
\begin{equation}\label{proof_thm_general_linear_9}
1 - \frac{M\lambda_t}{2} \geq e^{-M\lambda_t}, \qquad 0 \leq t \leq k.
\end{equation}
Hence, we can obtain that for $0 \leq t \leq k$,
\begin{equation}\label{proof_thm_general_linear_10}
\begin{split}
    \frac{1}{\xi_t}(1 - \frac{M\lambda_t}{2}) & = e^{-2\sum_{i = 0}^{t - 1}Mr_i}(1 - \frac{M\lambda_t}{2}) \geq e^{-2\sum_{i = 0}^{t - 1}Mr_i}e^{-M\lambda_t}\\
    & \geq e^{-2\sum_{i = 0}^{t - 1}M\lambda_i}e^{-M\lambda_t} \geq e^{-2M\sum_{i = 0}^{t}\lambda_i} \geq \frac{2}{3},
\end{split}
\end{equation}
where the equality holds due to the definition of \eqref{proof_thm_general_linear_2}, the first inequality holds due to \eqref{proof_thm_general_linear_9}, the second inequality holds due to \eqref{lemma_3_4}, the third inequality holds due to $M\lambda_k \geq 0$ and the last inequality holds due to \eqref{proof_thm_general_linear_8}. Substituting \eqref{proof_thm_general_linear_10} into \eqref{proof_thm_general_linear_7}, we get that
\begin{equation}\label{proof_thm_general_linear_11}
    \theta(J_t, G_t, x_{t + 1} - x_{t}) \leq 1 - \frac{\mu}{L\xi_t}(1 - \frac{M\lambda_t}{2}) \leq 1 - \frac{2\mu}{3L}, \qquad 0 \leq t \leq k.
\end{equation}
Therefore, \eqref{thm_general_linear_2} holds for $0 \leq t \leq k$. Now consider the case of $t = k + 1$. From \eqref{proof_thm_general_linear_6} for $t = k$ and the fact that ${(\ln{\frac{3}{2}})}/{8} < {1}/{16}$, we get that
\begin{equation}\label{proof_thm_general_linear_12}
    \frac{M\lambda_k}{2} \leq \frac{\ln{\frac{3}{2}}}{8}\frac{\mu}{L} \leq \frac{\mu}{16L}.
\end{equation}
From \eqref{lemma_general_theta_1} of Lemma~\ref{lemma_general_theta} and \eqref{lemma_3_4}, we have that
\begin{equation}\label{proof_thm_general_linear_13}
    \lambda_{k + 1} \leq (1 + \frac{Mr_k}{2})\theta(J_k, G_k, x_{k + 1} - x_{k})\lambda_k \leq (1 + \frac{M\lambda_k}{2})\theta(J_k, G_k, x_{k + 1} - x_{k})\lambda_k.
\end{equation}
Substituting \eqref{proof_thm_general_linear_11} for $t = k$ and \eqref{proof_thm_general_linear_12} into \eqref{proof_thm_general_linear_13}, we get that
\begin{equation}
    \lambda_{k + 1} \leq (1 + \frac{\mu}{16L})(1 - \frac{2\mu}{3L})\lambda_k = (1 - \frac{29\mu}{48L} - \frac{\mu^2}{24L^2})\lambda_k \leq (1 - \frac{29\mu}{48L})\lambda_k \leq (1 - \frac{\mu}{2L})\lambda_k.
\end{equation}
Thus, condition \eqref{thm_general_linear_3} holds for $t = k + 1$ since
\begin{equation}
\lambda_{k + 1} \leq (1 - \frac{\mu}{2L})\lambda_k \leq (1 - \frac{\mu}{2L})^{k + 1}\lambda_0.
\end{equation}
Using the same technique we can prove that condition \eqref{thm_general_linear_2} holds for $t = k + 1$. Therefore, we finish proving the conclusion \eqref{thm_general_linear_2} and \eqref{thm_general_linear_3} using induction. \hfill \qed

\section{Proof of Lemma~\ref{lemma_general_potential}}\label{sec:proof_of_lemma_general_potential}

For brevity, we use the following shorthanded notations
\begin{equation}
    c = \frac{\mu}{2dL}, \qquad \rho_t = 1 + \frac{M\lambda_{f}(x_t)}{2}, \qquad \alpha_t = \sigma_{t} + 4Md\lambda_t, \qquad \beta_t = \rho_{t}^4 (1 + 8Md\lambda_{t}).
\end{equation}
The initial condition \eqref{lemma_general_potential_1} indicates that $M\lambda_0 \leq C_0\mu/L \leq C_0 < 2$. Hence,  $r_t \leq \lambda_t$ of \eqref{lemma_3_4} in Lemma~\ref{lemma_3} always holds for any $t \geq 0$. Thus, we have that
\begin{equation}\label{proof_lemma_general_potential_1}
    1 + \frac{Mr_t}{2} \leq 1 + \frac{M\lambda_t}{2} = \rho_t, \qquad \forall t \geq 0.
\end{equation}
Substituting the above inequality into \eqref{lemma_2_2} and \eqref{lemma_2_3} of Lemma~\ref{lemma_2}, we obtain that
\begin{equation}\label{proof_lemma_general_potential_2}
     \frac{\nabla^2{f(x_t)}}{\rho_t} \preceq J_t \preceq \rho_t\nabla^2{f(x_t)}, \qquad  \frac{\nabla^2{f(x_{t+1})}}{\rho_t} \preceq J_t \preceq \rho_t\nabla^2{f(x_{t+1})}.
\end{equation}
From \eqref{proof_thm_general_linear_4} of the proof of Theorem~\ref{thm_general_linear}, we showed that for any $t \geq 0$, we have $\hat{G}_t \succeq \nabla^2{f(x_{t + 1})}$. Recall that $G_{t + 1} = BFGS(\nabla^2{f(x_{t + 1})}, \hat{G_t}, \bar{u})$ and $\bar{u} = \bar{u}(\nabla^2{f(x_{t + 1})}, \hat{G_t})$ in step $8$ and $9$ of Algorithm~\ref{algo_general}. Applying Lemma~\ref{lemma_BFGS_greedy}, we obtain that
\begin{equation}\label{proof_lemma_general_potential_3}
\sigma_{t + 1} \leq (1 - \frac{\mu}{dL})\sigma(\nabla^2{f(x_{t + 1})}, \hat{G}_t) \leq (1 - \frac{\mu}{2dL})\sigma(\nabla^2{f(x_{t + 1})}, \hat{G}_t) = (1 - c)\sigma(\nabla^2{f(x_{t + 1})}, \hat{G}_t).
\end{equation}
Using the condition $\hat{G}_t = (1 + \frac{Mr_t}{2})^2\bar{G}_t$ in step $7$ of Algorithm~\ref{algo_general} and \eqref{proof_lemma_general_potential_1}, we can observe that $\hat{G}_t \leq \rho_t^2\bar{G}_t$. Using this condition, \eqref{proof_lemma_general_potential_2} and the definition of $\sigma$ in \eqref{sigma}, we obtain
\begin{equation}\label{proof_lemma_general_potential_4}
    \sigma(\nabla^2{f(x_{t + 1})}, \hat{G}_t) = Tr(\nabla^2{f(x_{t + 1})}^{-1}\hat{G}_t) - d \leq \rho_t^2 Tr(\nabla^2{f(x_{t + 1})}^{-1}\bar{G}_t) - d \leq \rho_t^3 Tr(J_t^{-1}\bar{G}_t) - d
\end{equation}
From \eqref{proof_thm_general_linear_1}, we know that
\begin{equation}
\nabla^2{f(x_t)} \preceq G_t \preceq \xi_t \frac{L}{\mu}\nabla^2{f(x_t)}, \quad \forall t \geq 0.
\end{equation}
Combining the above inequality and \eqref{proof_lemma_general_potential_2}, we can show that,
\begin{equation}\label{proof_lemma_general_potential_5}
    \frac{1}{\rho_t}J_t \preceq G_t \preceq \xi_t\frac{L}{\mu}\rho_t J_t, \quad \forall t \geq 0.
\end{equation}
From \eqref{thm_general_linear_2} of Theorem~\ref{thm_general_linear}, we obtain that
\begin{equation}\label{proof_lemma_general_potential_6}
    \theta_t \leq 1- \frac{2\mu}{3L} \leq 1 \leq \rho_t.
\end{equation}
In summary, \eqref{proof_lemma_general_potential_5} shows that $G_t \succeq \frac{1}{\rho_t}J_t$ and \eqref{proof_lemma_general_potential_6} shows that $\theta_t \preceq \rho_t$. Consider \eqref{lemma_1_3} of Lemma~\ref{lemma_1} and take $G = G_t$, $A = J_t$, $G_{+} = BFGS(J_t, G_t, s_t) = \bar{G}_t$ in step 5 of Algorithm~\ref{algo_general} and $\xi = \rho_t$. Applying \eqref{lemma_1_3} of Lemma~\ref{lemma_1}, we obtain that
\begin{equation}
    \sigma(J_t, G_t) - \sigma(J_t, \bar{G}_t) \geq \frac{1}{4\rho_t^2}\theta_t^2 - \ln{\rho_t},
\end{equation}
which is equivalent to
\begin{equation}\label{proof_lemma_general_potential_7}
    Tr(J_t^{-1}\bar{G}_t) \leq Tr(J_t^{-1}G_t) - \frac{1}{4\rho_t^2}\theta_t^2 + \ln{\rho_t},
\end{equation}
where we use the definition of $\sigma$ in \eqref{sigma}. Substituting \eqref{proof_lemma_general_potential_7} into \eqref{proof_lemma_general_potential_4}, we obtain that
\begin{equation}\label{proof_lemma_general_potential_8}
    \sigma(\nabla^2{f(x_{t + 1})}, \hat{G}_t) \leq \rho_t^3 \left(Tr(J_t^{-1}G_t) - \frac{1}{4\rho_t^2}\theta_t^2 + \ln{\rho_t}\right) - d.
\end{equation}
Substituting \eqref{proof_lemma_general_potential_8} into \eqref{proof_lemma_general_potential_3}, we have that
\begin{equation}\label{proof_lemma_general_potential_9}
\sigma_{t + 1} \leq (1 - c)\left[\rho_t^3 \left(Tr(J_t^{-1}G_t) - \frac{1}{4\rho_t^2}\theta_t^2 + \ln{\rho_t}\right) - d\right].
\end{equation}
Applying \eqref{proof_lemma_general_potential_2} and the definition of $\sigma$ in \eqref{sigma} again, we obtain that
\begin{equation}\label{proof_lemma_general_potential_10}
    Tr(J_t^{-1}G_t) \leq \rho_t Tr(\nabla^2{f(x_t)}^{-1}G_t) = \rho_t (\sigma_t + d).
\end{equation}
Substituting \eqref{proof_lemma_general_potential_10} into \eqref{proof_lemma_general_potential_9}, we achieve that
\begin{equation}\label{proof_lemma_general_potential_11}
\begin{split}
    \sigma_{t + 1} & \leq (1 - c)\left[\rho_t^3 \left(\rho_t (\sigma_t + d) - \frac{1}{4\rho_t^2}\theta_t^2 + \ln{\rho_t}\right) - d\right]\\
    & = (1 - c)(\rho_t^4\sigma_t + \rho_t^4 d + \rho_t^3 \ln{\rho_t} - d) - \frac{1}{4}(1 - c)\rho_t\theta_t^2\\
    & \leq (1 - c)\rho_t^4(\sigma_t + d + \frac{1}{\rho_t}\ln{\rho_t} - \frac{1}{\rho_t^4}d) - \frac{1}{4}(1 - c)\theta_t^2,
\end{split}
\end{equation}
where the last inequality holds due to the condition $\rho_t \geq 1$. We have that
\begin{equation}\label{proof_lemma_general_potential_12}
\begin{split}
    d + \frac{1}{\rho_t}\ln{\rho_t} - \frac{1}{\rho_t^4}d & \leq d + \frac{d}{\rho_t}\ln{\rho_t} - \frac{1}{\rho_t^4}d = \frac{\rho_t^4 + \rho_t^3\ln{\rho_t} - 1}{\rho_t^4}d \leq (\rho_t^4 + \rho_t^3\ln{\rho_t} - 1)d\\
    & = \left[(1 + \frac{M\lambda_t}{2})^4 + (1 + \frac{M\lambda_t}{2})^3\ln{(1 + \frac{M\lambda_t}{2})} - 1\right]d\\
    & \leq (e^{2M\lambda_t} - 1 + \frac{M\lambda_t}{2}e^{\frac{3}{2}M\lambda_t})d,
\end{split}
\end{equation}
where the first inequality is due to $d \geq 1$, the second inequality is due to $\rho_t \geq 1$ and the last inequality holds due to $1 + x \leq e^x$. Since the initial condition \eqref{lemma_general_potential_1} holds, applying Theorem~\ref{thm_general_linear} we obtain that
\begin{equation}\label{proof_lemma_general_potential_13}
M\lambda_t \leq M\lambda_0 \leq C_0\frac{\mu}{L} \leq C_0 = \frac{\ln{\frac{3}{2}}}{4} \leq \frac{1}{8}.
\end{equation}
Hence, \eqref{proof_lemma_general_potential_12} can be upper bounded by
\begin{equation}\label{proof_lemma_general_potential_14}
\begin{split}
    d + \frac{1}{\rho_t}\ln{\rho_t} - \frac{1}{\rho_t^4}d & \leq (e^{2M\lambda_t} - 1 + \frac{M\lambda_t}{2}e^{\frac{3}{2}M\lambda_t})d\\
    & \leq (2M\lambda_t + 4M^2\lambda_t^2 + \frac{M\lambda_t}{2}e^{\frac{3}{2}M\lambda_t})d\\
    & = (2 + 4M\lambda_t + \frac{1}{2}e^{\frac{3}{2}M\lambda_t})Md\lambda_t\\
    & \leq (2 + \frac{1}{2} + \frac{1}{2}e^{\frac{3}{16}})Md\lambda_t\\
    & \leq 4Md\lambda_t,
\end{split}
\end{equation}
where the second inequality is due to $e^x - 1 \leq x + x^2$ for $x \leq \frac{1}{4}$ and the third inequality is due to \eqref{proof_lemma_general_potential_13}. Substituting \eqref{proof_lemma_general_potential_14} into \eqref{proof_lemma_general_potential_11}, we reach that
\begin{equation}
    \sigma_{t + 1} \leq (1 - c)\rho_t^4(\sigma_t + 4Md\lambda_t) - \frac{1}{4}(1 - c)\theta_t^2 = (1 - c)\left[(1 + \frac{M\lambda_t}{2})^4(\sigma_t + 4Md\lambda_t) - \frac{1}{4}\theta_t^2\right].
\end{equation}
This is equivalent to the conclusion \eqref{lemma_general_potential_2}. Now, we move forward to prove \eqref{lemma_general_potential_3}. Notice that \eqref{lemma_general_potential_2} is equivalent to
\begin{equation}
\sigma_{t} \leq (1 - c)\rho_{t - 1}^4(\sigma_{t - 1} + 4Md\lambda_{t - 1}) - \frac{1}{4}(1 - c)\theta_{t - 1}^2, \quad \forall t \geq 1.
\end{equation}
Recall the notation
\begin{equation}
    \alpha_t = \sigma_{t} + 4Md\lambda_t.
\end{equation}
Combining the above two conditions, we obtain that
\begin{equation}\label{proof_lemma_general_potential_15}
    \alpha_t \leq (1 - c)\rho_{t - 1}^4 \alpha_{t - 1} - \frac{1}{4}(1 - c)\theta_{t - 1}^2 + 4Md\lambda_t
\end{equation}
Notice that for any symmetric positive semi-definite matrices $A, B \in \mathbb{R}^{d \times d}$, we have
\begin{equation}
B \preceq Tr(A^{-1}B)A.
\end{equation}
From \eqref{proof_thm_general_linear_1} in the proof of Theorem~\ref{thm_general_linear}, we know that $G_t \succeq \nabla^2{f(x_t)}$. Taking $A = \nabla^2{f(x_t)}$ and $B = G_t - \nabla^2{f(x_t)}$ in the above inequality and using the definition of $\sigma$ in \eqref{sigma}, we get that
\begin{equation}
G_t - \nabla^2{f(x_t)} \preceq Tr(\nabla^2{f(x_t)}^{-1}(G_t - \nabla^2{f(x_t)}))\nabla^2{f(x_t)} = \sigma_t \nabla^2{f(x_t)}.
\end{equation}
Hence, we obtain that
\begin{equation}
\nabla^2{f(x_t)} \preceq G_t \preceq (1 + \sigma_t)\nabla^2{f(x_t)}.
\end{equation}
Applying \eqref{lemma_3_5} of Lemma~\ref{lemma_3} with $\eta = 1 +\sigma_t$, we obtain that
\begin{equation}
    \theta_t = \theta(J_t, G_t, x_{t + 1} - x_{t}) \leq \frac{\sigma_t + \frac{M\lambda_{t}}{2}}{1 + \sigma_t} \leq \sigma_t + \frac{M\lambda_{t}}{2} \leq \sigma_t + 4Md\lambda_{t},
\end{equation}
where the second inequality is due to $\sigma_t \geq 0$ and the third inequality holds due to $d \geq 1$. Combing \eqref{lemma_general_theta_1} of Lemma~\ref{lemma_general_theta}, \eqref{lemma_3_4} and the above inequality, we have that
\begin{equation}
    \lambda_{t + 1} \leq (1 + \frac{Mr_t}{2})\theta_t\lambda_t \leq (1 + \frac{M\lambda_t}{2})\theta_t\lambda_t \leq (1 + \frac{M\lambda_t}{2})(\sigma_t + 4Md\lambda_{t})\lambda_t = \rho_t \alpha_t \lambda_{t},
\end{equation}
Thus, we prove that
\begin{equation}\label{proof_lemma_general_potential_16}
    \lambda_t \leq \rho_{t - 1} \alpha_{t - 1} \lambda_{t - 1} \quad \forall t \geq 1.
\end{equation}
Substituting \eqref{proof_lemma_general_potential_16} into \eqref{proof_lemma_general_potential_15}, we have that
\begin{equation}
\begin{split}
    \alpha_t & \leq (1 - c)\rho_{t - 1}^4 \alpha_{t - 1} + 4Md\rho_{t - 1} \alpha_{t - 1} \lambda_{t - 1} - \frac{1}{4}(1 - c)\theta_{t - 1}^2\\
    & \leq (1 - c)\rho_{t - 1}^4 \alpha_{t - 1} + 8(1 - c)Md\rho_{t - 1}^4 \alpha_{t - 1} \lambda_{t - 1} - \frac{1}{4}(1 - c)\theta_{t - 1}^2\\
    & = (1 - c)\rho_{t - 1}^4 \alpha_{t - 1}(1 + 8Md\lambda_{t - 1}) - \frac{1}{4}(1 - c)\theta_{t - 1}^2,
\end{split}
\end{equation}
where the second inequality is due to $\frac{1}{2} \leq 1 - \frac{\mu}{2L} = 1 - c$ and $\rho_{t - 1} \geq 1$. Recall the notation $\beta_t = \rho_{t}^4 (1 + 8Md\lambda_{t})$. The above inequality can be simplified as
\begin{equation}
\alpha_t \leq (1 - c)\beta_{t - 1} \alpha_{t - 1} - \frac{1}{4}(1 - c)\theta_{t - 1}^2.
\end{equation}
Applying the above inequality recursively, we obtain the following result
\begin{equation}
\begin{split}
    \alpha_t & \leq (1 - c) \beta_{t - 1} \alpha_{t - 1} - \frac{1}{4}(1 - c)\theta_{t - 1}^2\\
    & \leq (1 - c)^2\beta_{t - 2}\beta_{t - 1} \alpha_{t - 2} - \frac{1}{4}(1 - c)^2\beta_{t - 1}\theta_{t - 2}^2 - \frac{1}{4}(1 - c)\theta_{t - 1}^2\\
    & \leq (1 - c)^t\alpha_0\prod_{j = 0}^{t - 1}\beta_j - \frac{1}{4}\sum_{i = 0}^{t - 1}(1 - c)^{t - i}\theta_{i}^2\prod_{j = i + 1}^{t - 1}\beta_j.
\end{split}
\end{equation}
Here we regulate that $\prod_{j = t}^{t - 1}\beta_j$ is $1$. The above inequality indicates that
\begin{equation}\label{proof_lemma_general_potential_17}
    \frac{1}{4}\sum_{i = 0}^{t - 1}(1 - c)^{t - i}\theta_{i}^2\prod_{j = i + 1}^{t - 1}\beta_j \leq (1 - c)^t\alpha_0\prod_{j = 0}^{t - 1}\beta_j - \alpha_t \leq (1 - c)^t\alpha_0\prod_{j = 0}^{t - 1}\beta_j.
\end{equation}
Since $\beta_j = \rho_{j}^4 (1 + 6Md\lambda_{j}) \geq 1$ for all $j \geq 1$, we obtain that
\begin{equation}\label{proof_lemma_general_potential_18}
    \prod_{j = i + 1}^{t - 1}\beta_j \geq 1, \qquad 0 \leq i \leq t - 1.
\end{equation}
Applying $1 + x \leq e^x$, we obtain that
\begin{equation}
\beta_j = (1 + \frac{M\lambda_j}{2})^4(1 + 8Md\lambda_j) \leq e^{2M\lambda_j}e^{8Md\lambda_j} = e^{10Md\lambda_j}, \quad \forall j \geq 0.
\end{equation}
Hence, from the linear convergence result of \eqref{thm_general_linear_3} and the initial condition \eqref{lemma_general_potential_1}, we observe
\begin{equation}\label{proof_lemma_general_potential_19}
    \prod_{j = 0}^{t - 1}\beta_j \leq \prod_{j = 0}^{t - 1}e^{10Md\lambda_j} = e^{10Md\sum_{j = 0}^{t - 1}\lambda_j} \leq e^{10Md\lambda_0\sum_{j = 0}^{t - 1}(1 - \frac{\mu}{2L})^j} \leq e^{20Md\frac{L}{\mu}\lambda_0} \leq e^{20C_1} = e^{\ln{2}} = 2.
\end{equation}
Leveraging the results in \eqref{proof_lemma_general_potential_17}, \eqref{proof_lemma_general_potential_18} and \eqref{proof_lemma_general_potential_19}, we obtain that
\begin{equation}
    \frac{1}{4}\sum_{i = 0}^{t - 1}(1 - c)^{t - i}\theta_{i}^2 \leq \frac{1}{4}\sum_{i = 0}^{t - 1}(1 - c)^{t - i}\theta_{i}^2\prod_{j = i + 1}^{t - 1}\beta_j \leq (1 - c)^t\alpha_0\prod_{j = 0}^{t - 1}\beta_j \leq 2(1 - c)^t \alpha_0.
\end{equation}
This is equivalent to
\begin{equation}
    \sum_{i = 0}^{t - 1}(1 - c)^{t - i}\theta_{i}^2 \leq 8(1 - c)^t \alpha_0.
\end{equation}
Dividing the term $(1 - c)^t$ on both sides of the above inequality, we can obtain that
\begin{equation}
\sum_{i = 0}^{t - 1}\frac{\theta^2_i}{(1 - c)^i} \leq 8\alpha_{0}.
\end{equation}
Hence, we prove the result \eqref{lemma_general_potential_3} since $c = \frac{\mu}{2dL}$ and $\alpha_0 = \sigma_0 + 4Md\lambda_0$. \hfill \qed

\section{Proof of Theorem~\ref{thm_general_superlinear}}\label{sec:proof_of_thm_general_superlinear}

Using $G_0 = LI$, initial condition \eqref{thm_general_superlinear_1}, the definition of $\sigma$ in \eqref{sigma} and Assumption~\ref{ass_str_cvx_smooth}, we obtain
\begin{equation}\label{proof_thm_general_superlinear_1}
    \sigma_0 + 4Md\lambda_0 = Tr(\nabla^2{f(x_0)}^{-1}G_0) - d + 4Md\lambda_0 \leq d\frac{L}{\mu} - d + 4\frac{\ln{2}}{20}\frac{\mu}{L} \leq d\frac{L}{\mu} - d + 1 \leq d\frac{L}{\mu}.
\end{equation}
Substituting \eqref{proof_thm_general_superlinear_1} into \eqref{lemma_general_potential_3}, we have that
\begin{equation}\label{proof_thm_general_superlinear_2}
\sum_{i = 0}^{t - 1}\frac{\theta^2_i}{(1 - c)^i} \leq 8d\frac{L}{\mu}.
\end{equation}
Using Lemma~\ref{lemma_general_theta} and \eqref{lemma_3_5} of Lemma~\ref{lemma_3} and recalling the notation $\rho_t = 1 + \frac{M\lambda_t}{2}$, we obtain that
\begin{equation}\label{proof_thm_general_superlinear_3}
\frac{\lambda_t}{\lambda_0} = \prod_{i = 0}^{t - 1}\frac{\lambda_{i + 1}}{\lambda_i} \leq \prod_{i = 0}^{t - 1}(1 + \frac{Mr_i}{2})\theta_i \leq \prod_{i = 0}^{t - 1}(1 + \frac{M\lambda_i}{2})\theta_i = \prod_{i = 0}^{t - 1}\rho_i \prod_{i = 0}^{t - 1}\theta_i.
\end{equation}
Applying $1 + x \leq e^x$, $d \geq 1$, the linear convergence result of \eqref{thm_general_linear_3} and the initial condition \eqref{thm_general_superlinear_1} again, we obtain that
\begin{equation}\label{proof_thm_general_superlinear_4}
\prod_{i = 0}^{t - 1}\rho_i = \prod_{i = 0}^{t - 1}(1 + \frac{M\lambda_i}{2}) \leq e^{\frac{M}{2}\sum_{i = 0}^{t - 1}\lambda_i} \leq e^{\frac{M}{2}\lambda_0\sum_{i = 0}^{t - 1}(1 - \frac{\mu}{2L})^{i}} \leq e^{\frac{M}{2}\lambda_0\frac{2L}{\mu}} \leq e^{\frac{C_1}{d}} \leq e^{C_1} = e^{\frac{ln{2}}{20}} \leq e^{\ln{2}} = 2.
\end{equation}
Leveraging \eqref{proof_thm_general_superlinear_3} and \eqref{proof_thm_general_superlinear_4}, we get that
\begin{equation}\label{proof_thm_general_superlinear_5}
    \frac{\lambda_t}{\lambda_0} \leq 2\prod_{i = 0}^{t - 1}\theta_i = 2\prod_{i = 0}^{t - 1}(1 - c)^{\frac{i}{2}}\frac{\theta_i}{(1 - c)^{\frac{i}{2}}} = 2\prod_{i = 0}^{t - 1}(1 - c)^{\frac{i}{2}}\prod_{i = 0}^{t - 1}\frac{\theta_i}{(1 - c)^{\frac{i}{2}}} = 2(1 - c)^{\frac{t(t - 1)}{4}}\prod_{i = 0}^{t - 1}\frac{\theta_i}{(1 - c)^{\frac{i}{2}}}.
\end{equation}
Using the arithmetic-geometric mean inequality and \eqref{proof_thm_general_superlinear_2}, we obtain that
\begin{equation}\label{proof_thm_general_superlinear_6}
    \prod_{i = 0}^{t - 1}\frac{\theta_i}{(1 - c)^{\frac{i}{2}}} = \left[\prod_{i = 0}^{t - 1}\frac{\theta^2_i}{(1 - c)^i}\right]^{\frac{1}{2}} \leq \left[\frac{1}{t}\sum_{i = 0}^{t - 1}\frac{\theta^2_i}{(1 - c)^i}\right]^{\frac{t}{2}} \leq \left(\frac{8dL}{\mu t}\right)^{\frac{t}{2}}.
\end{equation}
Combining \eqref{proof_thm_general_superlinear_5}, \eqref{proof_thm_general_superlinear_6} and $c = \frac{\mu}{2dL}$, we achieve the final convergence rate of \eqref{thm_general_superlinear_2}
\begin{equation}
\frac{\lambda_t}{\lambda_0} \leq 2(1 - c)^{\frac{t(t - 1)}{4}}\left(\frac{8dL}{\mu t}\right)^{\frac{t}{2}} = 2(1 - \frac{\mu}{2dL})^{\frac{t(t - 1)}{4}} (\frac{8dL}{t\mu})^{\frac{t}{2}}, \qquad \forall t \geq 1.
\end{equation}\hfill \qed

\section{Randomized Sharpened-BFGS Algorithm}\label{sec:random}

In this section, we extend our analysis to the randomized version of Sharpened-BFGS method. This is enlightened by the latest work of \cite{zhangzhihua2021quasinewton1}, where the authors proposed the modified Greedy-BFGS method based on the Cholesky factorization of the inverse Hessian approximation matrix. They presented that instead of selecting the greedy direction defined in \eqref{greedy_vector} of Lemma~\ref{lemma_BFGS_greedy}, we consider the following Greedy-BFGS update $G_+ = BFGS(A, G, R\bar{u}(A, R))$, where $R$ is the upper triangular matrix satisfying $A^{-1} = R^\top R$ and $\bar{u}(A, R)$ is defined as
\begin{equation}\label{greedy_modified}
    \bar{u}(A, R) := \argmax_{u \in \{e_i\}_{i = 1}^{d}} \frac{u^\top R^{-\top} A^{-1}R^{-1} u}{u^\top u}.
\end{equation}
Then, the linear convergence rate of $1 - 1/(d\kappa)$ in \eqref{lemma_BFGS_greedy_1} of Lemma~\ref{lemma_BFGS_greedy} can be improved to $1 - 1/d$, which is independent of the condition number $\kappa = L/\mu$. However, for each unit vector $e_i$ the computational cost of the term $e_i^\top R^{-\top} A^{-1}R^{-1} e_i$ is $\mathcal{O}(d^2)$. Hence, the cost of calculating the vector $\bar{u}(A, R)$ in \eqref{greedy_modified} is $\mathcal{O}(d^3)$, which makes this modified greedy update impractical to implement. Therefore, \cite{zhangzhihua2021quasinewton1} proposed to replace the greedy vector in \eqref{greedy_modified} by the random vector $\tilde{u} \sim \mathcal{N}(0, I_d)$ and consider the randomized BFGS update $\bar{G}_+ = BFGS(A, G, R\tilde{u})$, where $R$ is still the upper triangular Cholesky factorization matrix of $A^{-1}$. The condition-number-free linear convergence rate of $1 - 1/d$ is preserved for this randomized algorithm. This is summarized in the following lemma.

\begin{lemma}[\cite{zhangzhihua2021quasinewton1}]\label{lemma_BFGS_random}
Consider positive definite matrices $A, G \in \mathbb{R}^{d \times d}$ that satisfy $A \preceq G$. Suppose that $\bar{G}_{+} = BFGS(A, G, R^\top\tilde{u})$ where $R$ is the upper triangular matrix with $G^{-1} = R^\top R$ and $\tilde{u} \sim \mathcal{N}(0, I_d) \in \mathbb{R}^d$ is the random vector. Then, we have
\begin{equation}\label{lemma_BFGS_random_1}
    \mathbb{E}\left[\sigma(A, \bar{G}_{+})\right] \leq \left(1 - \frac{1}{d}\right)\mathbb{E}\left[\sigma(A, G)\right].
\end{equation}
\end{lemma}

Notice that the computational cost of Cholesky decomposition is in general $\mathcal{O}(d^3)$ for a matrix with dimension $d$. However, the expense per iteration of randomized BFGS method could be reduced to $\mathcal{O}(d^2)$ using technique highlighted in \cite{zhangzhihua2021quasinewton1}. Therefore, we can improve the superlinear convergence rate of our Sharpened-BFGS algorithm by replacing the Greedy-BFGS update with the randomized BFGS method proposed in \cite{zhangzhihua2021quasinewton1}. Meanwhile, the computational cost per iteration of randomized Sharpened-BFGS method is still $\mathcal{O}(d^2)$. This novel randomized Sharpened-BFGS method is summarized in Algorithm~\ref{algo_random}. The local linear convergence rate presented in Theorem~\ref{thm_general_linear} still holds for this randomized Sharpened-BFGS method. In the following theorem, we directly show the explicit local superlinear convergence rate for this randomized Sharpened-BFGS algorithm.

\begin{theorem}\label{thm_random}
Consider the randomized Sharpened-BFGS quasi-Newton method in Algorithm~\ref{algo_random} applied to the objective function satisfying Assumption~\ref{ass_str_cvx_smooth} and \ref{ass_str_concordant}. Suppose that the initial point $x_0$ satisfies that
\begin{equation}\label{thm_random_1}
\lambda_0 \leq \frac{C_1\mu}{dML}, \quad C_1 = \frac{\ln{2}}{20}.
\end{equation}
Then, we can reach the following local superlinear convergence rate with high probability
\begin{equation}\label{thm_random_2}
    \lambda_t \leq 2(1 - \frac{1}{2d})^{\frac{t(t - 1)}{4}} (\frac{8dL}{t\mu})^{\frac{t}{2}}\lambda_0, \qquad \forall t \geq 1.
\end{equation}
\end{theorem}

\begin{proof}
Here we just present the abbreviated proof to avoid repeated details since the proof of this theorem is very similar to the proof of Lemma~\ref{lemma_general_potential} and Theorem~\ref{thm_general_superlinear}. From the theory of probability, Lemma~\ref{lemma_BFGS_random} shows that there exists a constant $\delta$ such that the inequality
\begin{equation}
    \sigma(A, \bar{G}_{+}) \leq (1 - \frac{1}{d})\sigma(A, G)
\end{equation}
holds with probability at least $1 - \delta$. Here we neglect this parameter $\delta$ to simplify the proof and denote that the above inequality holds with high probability. Then, applying the same techniques from the proof of Lemma~\ref{lemma_general_potential}, we can show that the following condition holds with high probability for any $t \geq 0$
\begin{equation}
    \sigma_{t+1} \leq (1 - \frac{1}{2d})\left[(1 + \frac{M\lambda_t}{2})^4(\sigma_t + 4Md\lambda_t) - \frac{1}{4}\theta^2_t\right],
\end{equation}
where  $\theta_t := \theta(\nabla^2{f(x_{t})}, G_t, x_{t + 1} - x_{t})$ and $\sigma_t := \sigma(\nabla^2{f(x_{t})}, G_{t})$. Moreover, we have that with high probability
\begin{equation}
    \sum_{i = 0}^{t - 1}\frac{\theta^2_i}{(1 - \frac{1}{2d})^{i}}  \leq 8(\sigma_0 + 4Md\lambda_0), \qquad \forall t \geq 1. 
\end{equation}
Finally, using the same methods from the proof of Theorem~\ref{thm_general_superlinear}, we can prove that the suplinear convergence rate of \eqref{thm_random_2} holds with high probability.
\end{proof}

We observe that the quadratic convergence rate term is $\mathcal{O}((1 - 1/d)^{t^2})$ in the above superlinear convergence rate in \eqref{thm_random_2}, which is independent of the condition number $\kappa$. This condition-number-free quadratic convergence rate is the direct consequence of the linear convergence rate of \eqref{lemma_BFGS_random_1} from Lemma~\ref{lemma_BFGS_random}.

\begin{algorithm}[t]
\caption{The randomized Sharpened-BFGS method.}\label{algo_random} 
\begin{algorithmic}[1] 
{\REQUIRE Initial point $x_0$ and initial Hessian approximation matrix $G_0 = LI$.
\FOR {$t = 0,1,2,\ldots$}
    \STATE Update the variable: $x_{t + 1} = x_t - G_t^{-1} \nabla{f(x_t)}$;
    \STATE Compute the variable difference: $s_t = x_{t + 1} - x_t$;
    \STATE Set the matrix: $J_t = \int_{0}^{1}\nabla^2{f(x_t + \tau s_t)}d\tau$;
    \STATE Compute the matrix: $\bar{G_t} = BFGS(J_t, G_t, s_t)$;
    \STATE Compute the correction term: $r_t = \|x_{t + 1} - x_{t}\|_{x_t}$;
    \STATE Compute the matrix: $\hat{G_t} = (1 + {Mr_t}/{2})^2\bar{G_t}$;
    \STATE Compute upper triangular matrix: $R_t$ with $\hat{G_t}^{-1} = R_t^\top R_t$;
    \STATE Choose the random direction: $\tilde{u} \sim \mathcal{N}(0, I_d)$;
    \STATE Compute $G_{t + 1} = BFGS(\nabla^2{f(x_{t + 1})}, \hat{G_t}, R_t^\top \tilde{u})$;
\ENDFOR}
\end{algorithmic}
\end{algorithm}

\section*{Acknowledgement}
This research of Q. Jin and A. Mokhtari is supported in part by NSF Grants 2007668, 2019844, and 2112471, ARO Grant W911NF2110226, the Machine Learning Lab (MLL) at UT Austin, and the Wireless Networking and Communications Group (WNCG) Industrial Affiliates Program.  

\newpage

{{
\bibliography{bmc_article.bib}
\bibliographystyle{abbrvnat}
}}

\end{document}